\documentclass[12pt]{amsart}
\usepackage[all]{xy} 
\usepackage{amsfonts, amssymb, MnSymbol, mathdots, tikz-cd} 
\usepackage{geometry}
\newtheorem{thm}[equation]{Theorem}
\newtheorem{cor}[equation]{Corollary}
\newtheorem{prop}[equation]{Proposition}
\theoremstyle{definition}

\newtheorem{lem}[equation]{Lemma}
\newtheorem{defn}[equation]{Definition}
\newtheorem{ex}[equation]{Example}
\newtheorem{exs}[equation]{Examples}

\newtheorem{obs}[equation]{Observation}

\newtheorem*{ack}{Acknowledgments}

\allowdisplaybreaks

\numberwithin{equation}{section}

\let\realsubsection\subsection% keep a copy
\def\subsection{\setcounter{subsection}{\arabic{equation}}% sync the counters
   \refstepcounter{equation}%
   \realsubsection% invoke the real \section
}

\newcommand{\Z}{\mathbb Z}
\newcommand{\CSH}{\mathbb H}

\newcommand{\bu}{\bullet}

\newcommand{\al}{\alpha}
\newcommand{\be}{\beta}
\newcommand{\ep}{\varepsilon}
\newcommand{\wt}{\widetilde}
\newcommand{\ot}{\otimes}
\newcommand{\De}{\Delta}
\newcommand{\CB}{\mathcal{B}}
\newcommand{\CZ}{\mathcal{Z}}
\newcommand{\CF}{{\mathcal{F}}}
\newcommand{\delt}{\mathcal{D}}
\newcommand{\NCH}{\overline{CH}}
\newcommand{\B}[1]{\textit{\textbf{#1}}}
\newcommand{\BF}{\textit{\textbf{F}}}
\newcommand{\BM}{\textit{\textbf{M}}}
\newcommand{\ov}[1]{_{\hspace{-.2mm}/\hspace{-.5mm}#1}}
\newcommand{\un}[1]{\underline{#1}}
\newcommand{\eps}[2]{|\underline{a_{#1}}|+\dots+|\underline{a_{#2}}|}
\newcommand{\gen}[1]{\langle #1 \rangle}
\newcommand{\hip}[4]{ [ \hspace{1mm}#1 \hspace{1mm}|\hspace{1mm}#2\hspace{1mm}|\hspace{1mm}#3\hspace{1mm}|\hspace{1mm}#4\hspace{1mm}]}
\newcommand{\rep}[1]{\underset{ \widetilde{\hspace{5mm}}}{#1}}
\newcommand{\cancl}[1]{\scalebox{0.6}{\circled{#1}}}
\newcommand*\circled[1]{\tikz[baseline=(char.base)]{
            \node[shape=rectangle,draw, fill=gray!20] (char) {\bf #1};}}

\title[A note on the string topology BV-algebra for $S^2$ with $\Z_2$ coefficients]{A note on the string topology BV-algebra for $S^2$ with $\Z_2$ coefficients}

\author[K.~Poirier]{Kate~Poirier}
  \address{Kate Poirier,
  Department of Mathematics, New York City College of Technology, City University of New York, 300 Jay Street, Brooklyn, NY 11201}
  \email{kpoirier@citytech.cuny.edu}

\author[T.~Tradler]{Thomas~Tradler}
  \address{Thomas Tradler,
  Department of Mathematics, New York City College of Technology, City University of New York, 300 Jay Street, Brooklyn, NY 11201}
  \email{ttradler@citytech.cuny.edu}


\begin{document}
\maketitle
\begin{abstract}
Luc Menichi showed that the BV algebras on  $H_\bu(L S^2;\Z_2)[-2]$ coming from string topology and the one on $HH^\bu(H^\bu( S^2;\Z_2), H^\bu( S^2;\Z_2))$ using Poincar\'e duality on $H^\bu( S^2;\Z_2)$ are not isomorphic. In this note we show how one can obtain the string topology BV algebra on Hochschild cohomology using a Poincar\'e duality structure with higher homotopies. This Poincar\'e duality (with higher homotopies) on cohomology is induced by a local Poincar\'e duality (with higher homotopies) on the cochain level.
\end{abstract}

\section{Main Statements}\label{SEC:Introduction}

\subsection{BV algebras for the $2$-sphere with $\Z_2$ coefficients} 
In \cite[Theorem 24]{M1} Luc Menichi calculated the string topology BV algebra (defined by Moira Chas and Dennis Sullivan in \cite{CS}) for the $2$-sphere $ M= S^2$ with $\Z_2=\Z/2\Z$ coefficients to be
\begin{align}\label{EQU:CH(LS2,Z2)}
\CSH_\bu(L S^2;\Z_2)&\cong \Lambda a\otimes \Z_2[u]\cong \bigoplus_{k\geq 0} \Z_2.\al_k\oplus \bigoplus_{k\geq 0} \Z_2.\be_k,\\
\nonumber  &\text{ where } \al_k\cdot \al_\ell=\al_{k+\ell}, \be_k\cdot \be_\ell=0, \text{ and } \al_k\cdot \be_\ell=\be_\ell\cdot \al_k=\be_{k+\ell},\\
\nonumber &\text{ and } \Delta_{\text{ST}}(\al_k)=0 \text{ and } \Delta_{\text{ST}}(\be_k)=k\cdot \al_{k-1}+k\cdot \be_{k+1}.
\end{align}
Here, the degrees are given by $|a|=-2$, $|u|=1$, and thus, setting $\al_k=1\otimes u^k$ and $\be_k=a\otimes u^k$, we have $|\al_k|=k$, $|\be_k|=k-2$.

Moreover, in \cite[Proposition 20]{M1}, Menichi calculated the BV algebra for Hochschild cohomology of the cohomology of $S^2$ with $\Z_2$ coefficients to be
\begin{align}\label{EQU:HH(H(S2,Z2))}
HH^\bu(H^\bu(S^2;\Z_2); H^\bu(S^2;\Z_2))&\cong \Lambda g\otimes \Z_2[f]\cong \bigoplus_{k\geq 0} \Z_2.\phi_k\oplus \bigoplus_{k\geq 0} \Z_2.\psi_k,\\
\nonumber  &\hspace{-2.8cm}\text{ where } \phi_k\cdot \phi_\ell=\phi_{k+\ell}, \psi_k\cdot \psi_\ell=0, \text{ and } \phi_k\cdot \psi_\ell=\psi_\ell\cdot \phi_k=\psi_{k+\ell},\\
\nonumber &\hspace{-2.8cm}\text{ and } \Delta(\phi_k)=0 \text{ and } \Delta(\psi_k)=k\cdot \phi_{k-1}.
\end{align}
There degrees are similarly given by $|g|=-2$, $|f|=1$, and so, setting $\phi_k=1\otimes f^k$ and $\psi_k=g\otimes f^k$, we have $|\phi_k|=k$, $|\psi_k|=k-2$.

Note, that these BV algebras differ in their $\Delta$ operators, and, in fact, Menichi obtained the following result.
\begin{thm}\cite[Corollary 30]{M1}\label{THM:Menichi-no-BV-iso}
There is no isomorphism of BV algebras between\linebreak $(\CSH_\bu(LS^2;\Z_2),\cdot, \Delta_{\text{\emph{ST}}})$ and $(HH^\bu(H^\bu(S^2;\Z_2); H^\bu(S^2;\Z_2)),\cdot,\Delta)$.
\end{thm}
It is worth noting though, that the induced Gerstenhaber algebras on $\CSH_\bu(LS^2;\Z_2)$ and $HH^\bu(H^\bu(S^2;\Z_2); H^\bu(S^2;\Z_2))$ are isomorphic; cf. \cite[Corollary 23]{M1}.

\subsection{The BV algebra coming from local Poincar\'e duality}
A crucial ingredient for determining the $\Delta$ operator on Hochschild cohomology \eqref{EQU:HH(H(S2,Z2))} comes from a choice of Poincar\'e duality structure given as a bimodule isomorphism $F:H^\bu(S^2;\Z_2)\stackrel\cong\longrightarrow H_\bu(S^2;\Z_2)$. For the ($\Z_2$-)cohomology $H^\bu(S^2;\Z_2)$ and homology $H_\bu(S^2;\Z_2)$, an obvious choice is to define the degree $+2$ isomorphism $F$ as follows:
\[
\begin{tikzcd}[column sep=small]
& & \text{\tiny degree $-2$} & \text{\tiny degree $-1$} & \text{\tiny degree $0$} & \text{\tiny degree $1$} & \text{\tiny degree $2$} &  \\ 
H^\bu(S^2;\Z_2)\arrow{d}{F} &=\dots & \Z_2 \arrow{l}\arrow{drr}{F} & 0 \arrow{l} &\Z_2\arrow{l}\arrow{l}\arrow{drr}{F} &0\arrow{l} &0\arrow{l} & \dots \arrow{l} \\
H_\bu(S^2;\Z_2)&=\dots & 0\arrow{l} & 0\arrow{l} &\Z_2\arrow{l} &0\arrow{l} &\Z_2\arrow{l} & \dots \arrow{l}
\end{tikzcd}
\]
Indeed, this choice of $F$, which is the map given by capping with the $\Z_2$-fundamental class of $S^2$, leads to the BV algebra in \eqref{EQU:HH(H(S2,Z2))} as we will confirm in theorem \ref{THM:main-theorem} (together with \ref{SUBSEC:Delta-F-strict}) below.

Note that on the cochain level, capping with a fundamental cycle is not a bimodule map; see observation \ref{OBS:not-bimodule}. Using a generalization of bimodule maps, which, in addition to $F$, allows for higher homotopies, we show in this paper that one \emph{can} obtain the string topology BV algebra \eqref{EQU:CH(LS2,Z2)} on Hochschild cohomology. The precise definition of bimodule maps with higher homotopies was studied in \cite{T1} and will be reviewed in section \ref{SEC:Hochschild} below.

Applying the concept of bimodule maps with higher homotopies, we will construct a specific example of such a map for the case of a local cochain model of the $2$-sphere $C^\bu(S^2;\Z_2)$ in example \ref{EX:local-CS2Z2}. Then we pull this map back to obtain a bimodule map with homotopies $\wt{F}$ for cohomology $H^\bu(S^2;\Z_2)$ (see proposition \ref{PROP:pullback-of-local}). The resulting map $\wt{F}$, recorded in example \ref{EX:HS2Z2-with-homotopies}, has precisely one higher homotopy, when compared to $F$ (see \eqref{EQU:HIP-in-H(S2,Z2)}). Now, using this $\wt F$, we obtain the following main theorem.
\begin{thm}\label{THM:main-theorem}
The BV algebra on $HH^\bu(H^\bu(S^2;\Z_2); H^\bu(S^2;\Z_2))$ induced by $\wt F$ (coming from a local bimodule map with higher homotopies transferred to cohomology) is isomorphic to the string topology BV algebra \eqref{EQU:CH(LS2,Z2)} on $\CSH_\bu(LS^2;\Z_2)$.
\end{thm}
\begin{proof}
We compute both BV algebras coming from $F$ and $\wt F$, respectively. Denote by $H^\bu:=H^\bu(S^2;\Z_2)$ and $H_\bu:=H_\bu(S^2;\Z_2)$. Then the bimodule map $F$ and the bimodule map $\wt F$ (with higher homotopies) induce respective (graded module) isomorphisms $\CF$ and $\wt{\CF}$ of Hochschild cohomologies (see \ref{SUBSEC:Delta-F-strict} and \ref{SUBSEC:Delta-F-homotopies})
\begin{equation}\label{EQU:Hoch-bimodule-map-H}
\CF, \wt{\CF}:
HH^\bu(H^\bu,H^\bu)\stackrel\cong\longrightarrow HH^\bu(H^\bu,H_\bu).
\end{equation}
Next, we use the explicit description (as reviewed in \ref{SUBSEC:HH(H,H)-and-cup} and \ref{SUBSEC:HH(H,H*)-and-B}) stating that
\begin{align*}
HH^\bu(H^\bu,H^\bu)\cong\bigoplus_{k\geq 0} \Z_2.\phi_k\oplus \bigoplus_{k\geq 0} \Z_2.\psi_k,
\quad\text{ where $|\phi_k|=k$, $|\psi_k|=k-2$,}
\\
HH^\bu(H^\bu,H_\bu)\cong\bigoplus_{k\geq 0} \Z_2.\theta_k\oplus \bigoplus_{k\geq 0} \Z_2.\chi_k,
\quad\text{ where $|\theta_k|=k+2$, $|\chi_k|=k$.}
\end{align*}
With this notation, $F$ induces in \eqref{EQU:Hoch-bimodule-map-H} the map $\CF(\phi_k)=\theta_k, \CF(\psi_k)=\chi_k$ (see \ref{SUBSEC:Delta-F-strict}), while $\wt{F}$ induces the map $\wt{\CF}(\phi_k)=\theta_k+\chi_{k+2}, \wt{\CF}(\psi_k)= \chi_k$ (computed in \ref{SUBSEC:Delta-F-homotopies}).

The BV $\Delta$ operator on $HH^\bu(H^\bu,H^\bu)$ is given by transferring Connes' $\CB$-operator from $HH^\bu(H^\bu,H_\bu)$ to $HH^\bu(H^\bu,H^\bu)$ via the given Poincar\'e duality bimodule isomorphism, where, Connes' $\CB$-operator is the map (computed in \ref{SUBSEC:HH(H,H*)-and-B}):
\[
\CB:HH^\bu(H^\bu,H_\bu)\to HH^\bu(H^\bu,H_\bu), \quad \CB( \theta_k)=0, \CB(\chi_k)=k\cdot \theta_{k-1}.
\]
Thus the induced operators $\Delta=\CF^{-1}\circ \CB\circ \CF$ and $\wt\Delta=\wt{\CF}^{-1}\circ \CB\circ \wt{\CF}$ from $F$ and $\wt F$, respectively, are:
\begin{align}
&\label{EQU:Delta-on-HH}
\Delta(\phi_k)=0, \quad\quad\quad \Delta(\psi_k)=\CF^{-1}(k\cdot \theta_{k-1})=k\cdot \phi_{k-1},
\\
&\label{EQU:tilde-Delta-on-HH}
\wt \Delta(\phi_{k-2})= \wt \Delta(\psi_{k})=\wt{\CF}^{-1}(k\cdot \theta_{k-1})=k\cdot(\phi_{k-1}+\psi_{k+1}),
\end{align}
where we used that $\wt{\CF}^{-1}$ is given by $\wt{\CF}^{-1}(\theta_k)=\phi_k+\psi_{k+2}$ and $\wt{\CF}^{-1}(\chi_k)=\psi_k$.

Note that the $\Delta$ operator in \eqref{EQU:Delta-on-HH} coincides with $\Delta$ from \eqref{EQU:HH(H(S2,Z2))}. To see that the BV algebra with $\wt\Delta$ is isomorphic to the string topology BV algebra \eqref{EQU:CH(LS2,Z2)}, we use an isomorphism which was essentially described in \cite[last paragraph of Lemma 21]{M1}: the map $\Theta:HH^\bu(H^\bu,H^\bu)\to \CSH_\bu(L S^2;\Z_2)$ with $\Theta(\phi_k)=\al_k+k\cdot \be_{k+2}$ and $\Theta(\psi_k)=\be_k$ is an algebra isomorphism such that $\Delta_{\text{ST}}\circ \Theta=\Theta\circ \wt\Delta$.
\end{proof}

\subsection{Organization of the paper}
In section \ref{SEC:Hochschild}, we review basics about Hochschild cochain complexes and morphisms between them coming from bimodule maps up to higher homotopies. In section \ref{SEC:BV-on-Hochschild} we review the induced BV algebra coming from a chosen Poincar\'e duality structure. In section \ref{SEC:compute-bimodule-maps} we give various computations of bimodule maps (with and without higher homotopies). In particular, the construction of a bimodule map (up to certain controlled structures) for the $2$-simplex coming from locality (see example \ref{EX:F-for-the-2-simplex}) constitutes the main computational aspect of this paper, with its proof being spelled out in appendix \ref{APP:Proof-of-HIP-on-2-simplex}. In section \ref{SEC:compute-Hochschild-BV} we compute the BV algebras coming from the Poincar\'e duality structures on cohomology of the $2$-sphere from section \ref{SEC:compute-bimodule-maps}, and with this complete the proof of theorem \ref{THM:main-theorem}.

\begin{ack}
We would like to thank Mahmoud Zeinalian for discussions about this topic. The second author was partially supported by a PSC-CUNY research award.
\end{ack}

\section{Bimodule maps with higher homotopies}\label{SEC:Hochschild}

In this section we review bimodule maps, bimodule maps up to higher homotopies, and homotopy inner products; see also \cite{T1}. Moreover we review the induced maps on Hochschild cohomology.

\subsection{Basic setup}\label{SUBSEC:basic-setup}
Denote by $R$ a commutative ring with unit; our main example of interest is $R=\Z_2$. Let $A$ be a unital dg-algebra over $R$, where we note that all the differentials in this paper will always go down, i.e., $d:A_j\to A_{j-1}$. We denote by $\un{A}$ the space $A$ shifted up by one, i.e., $\un{A}_j:=A_{j-1}$.

Let $\BM$ be a dg-bimodule over $A$. If we want to emphasize the corresponding algebra $A$, then we will also write $\BM\ov{A}$ instead of $\BM$. From now on, all modules as well as module maps will be written in bold.

Note that if $\BM$ is a dg-bimodule over $A$, then so is its dual space $\BM^*$ with $\BM^*_j:=(\BM_{-j})^*=Hom_R(\BM_{-j},R)$ with differential $d:\BM^*_j\to \BM^*_{j-1}, d(n)(m):=(-1)^{|n|+1}n(d(m))$, and module maps $(n.a)(m):=n(a.m)$ and $(a.n)(m):=(-1)^{|a|\cdot(|n|+|m|)}n(m.a)$ for $n\in \BM^*, m\in \BM, a\in A$, where $|.|$ denotes the degree. Moreover, the dg-algebra $A$ is itself a dg-module $\B{A}:=A$ over $A$ (with module structure given by the algebra multiplication), and thus $\B{A}^*$ is a dg-module over $A$ as well.

Define the Hochschild cochain complex of $A$ with values in $\BM$ to be given by $CH^\bu(A, \BM):=\prod_{r\geq 0 } Hom(\un{A}^{\otimes r},\BM)$, where the differential $\delt=\delt_0+\delt_1$, with $\delt^2=0$, is defined for $\varphi\in Hom(\un{A}^{\otimes r},\BM)$ by setting
\begin{multline*}
\delt_0(\varphi)(\un{a_1},\dots,\un{a_r}):= d(\varphi(\un{a_1},\dots,\un{a_r}))
+\sum_{j=1}^r (-1)^{|\varphi|+\eps{1}{j-1}}\cdot \varphi(\un{a_1},\dots,\un{da_j},\dots,\un{a_r}),\\
\delt_1(\varphi)(\un{a_1},\dots,\un{a_{r+1}}):=(-1)^{|\varphi|\cdot |\un{a_1}|}\cdot a_1. \varphi(\un{a_2},\dots,\un{a_{r+1}})\hspace{7cm}\\
 +\sum_{j=1}^r (-1)^{|\varphi|+\eps{1}{j}}\cdot \varphi(\un{a_1},\dots,\un{a_j\cdot a_{j+1}},\dots,\un{a_{r+1}})
  +(-1)^{ |\varphi|+1+\eps{1}{r}}\cdot  \varphi(\un{a_1},\dots,\un{a_{r}}). a_{r+1}.
\end{multline*}
We will mainly use the normalized Hochschild cochain complex $\NCH^\bu(A, \BM)$, which is the subcomplex of $CH^\bu(A, \BM)$ consisting of those $\varphi\in CH^\bu(A, \BM)$ which vanish when any of the inputs is the unit $\un 1\in\un A$. The inclusion $\NCH^\bu(A, \BM)\hookrightarrow CH^\bu(A, \BM)$ is a quasi-isomorphism; see \cite[1.5.7]{L}.

\subsection{Inner products and homotopy inner products}\label{SUBSEC:IP-and-HIP}
Let $\BM$ and $\B{N}$ be dg-bimodules over $A$, and let $\BF:\BM\to \B{N}$ be a dg-bimodule map. Then, there is an induced cochain map $CH(\BF):\NCH^\bu(A,\BM)\to \NCH^\bu(A,\B{N})$, $\varphi\mapsto \BF\circ \varphi$. An inner product on $\BM$ is a dg-bimodule map $\BF:\BM\to \BM^*$.

We will need a more general notion of inner product $\BF:\BM\to \BM^*$ which also allows for higher homotopies. To this end, consider a sequence of maps $\BF=\{\BF_{p,q}:\un{A}^{\ot p}\ot \BM\ot \un{A}^{\ot q}\to \BM^*\}_{p,q\geq 0}$, and define its differential by setting $\delt \BF=\{(\delt \BF)_{p,q}\}_{p,q\geq 0}$ to be $(\delt \BF)_{p,q}=(\delt_0 \BF)_{p,q}+(\delt_1 \BF)_{p,q}$, where $\forall a_1,\dots, a_p,b_1,\dots, b_q,\in A$ and $m,n\in \BM$:
\begin{align}\label{EQU:de0(F)}
&(\delt_0 \BF)_{p,q}(\un{a_1},\dots, \un{a_p}; m;\un{b_1}, \dots, \un{b_q})(n) \\
\nonumber &  :=\sum_{j=1}^p (-1)^{|\BF|+|\un{a_1}|+ \dots+|\un{a_{j-1}}| } \cdot 
\BF_{p,q}(\un{a_1},\dots, \un{da_j}, \dots, \un{a_p};m;\un{b_1},\dots, \un{b_q})(n)\\
\nonumber & \quad + (-1)^{|\BF|+|\un{a_1}|+ \dots+|\un{a_{p}}|+1 } \cdot 
\BF_{p,q}(\un{a_1},\dots, \un{a_p};dm;\un{b_1},\dots, \un{b_q})(n)\\
\nonumber & \quad +\sum_{j=1}^q (-1)^{|\BF|+|\un{a_1}|+ \dots+|\un{a_{p}}|+|m|+|\un{b_1}|+ \dots+|\un{b_{j-1}}| }\cdot 
 \BF_{p,q}(\un{a_1},\dots, \un{a_p};m;\un{b_1},\dots, \un{db_j},\dots,\un{b_q})(n)\\
\nonumber & \quad + (-1)^{|\BF|+|\un{a_1}|+ \dots+|\un{a_{p}}|+|m|+|\un{b_1}|+ \dots+|\un{b_{q}}|+1 }\cdot 
 \BF_{p,q}(\un{a_1}, \dots, \un{a_p};m;\un{b_1},\dots, \un{b_q})(dn)
\end{align}
and, using the notation $\BF_{-1,q}=\BF_{p,-1}=0$:
\begin{align}\label{EQU:de1(F)}
& (\delt_1 \BF)_{p,q}(\un{a_1},\dots, \un{a_p}; m;\un{b_1}, \dots, \un{b_q})(n) \\
\nonumber &  := (-1)^{|\BF|+|a_1|\cdot(|\un{a_2}|+ \dots+|\un{a_{p}}|+|m|+|\un{b_1}|+ \dots+|\un{b_{q}}|+|n|) } \cdot 
\BF_{p-1,q}(\un{a_2},\dots, \un{a_p};m;\un{b_1},\dots, \un{b_q})(n. a_1)\\
\nonumber & \quad +\sum_{j=1}^{p-1} (-1)^{|\BF|+|\un{a_1}|+ \dots+|\un{a_{j}}| } \cdot \BF_{p-1,q}(\un{a_1},\dots, \un{a_j\cdot a_{j+1}},\dots,\un{a_p};m;\un{b_1},\dots, \un{b_q})(n)\\
\nonumber & \quad + (-1)^{|\BF|+|\un{a_1}|+ \dots+|\un{a_{p-1}}|+1 } \cdot 
\BF_{p-1,q}(\un{a_1},\dots, \un{a_{p-1}};a_p. m;\un{b_1},\dots, \un{b_q})(n)\\
\nonumber & \quad + (-1)^{|\BF|+|\un{a_1}|+ \dots+|\un{a_{p}}|+|m| } \cdot 
\BF_{p,q-1}(\un{a_1},\dots, \un{a_p};m. b_1;\un{b_2},\dots, \un{b_q})(n)\\
\nonumber & \quad +\sum_{j=1}^{q-1} (-1)^{|\BF|+|\un{a_1}|+ \dots+|\un{a_{p}}|+|m|+|\un{b_1}|+ \dots+|\un{b_{j}}| } \cdot 
\BF_{p,q-1}(\un{a_1},\dots, \un{a_p};m;\un{b_1},\dots, \un{b_j\cdot b_{j+1}},\dots, \un{b_q})(n)\\
\nonumber & \quad + (-1)^{|\BF|+|\un{a_1}|+ \dots+|\un{a_{p}}|+|m|+|\un{b_1}|+ \dots+|\un{b_{q-1}}|+1 } \cdot 
\BF_{p,q-1}(\un{a_1},\dots, \un{a_p};m;\un{b_1},\dots, \un{b_{q-1}})(b_q. n)
\end{align}
Then we call $\BF$ a homotopy inner product for $\BM$ if $\delt\BF=0$. By slight abuse of notation, we will still use the notation $\BF:\B{M}\to \B{M}^*$ for $\BF=\{\BF_{p,q}\}_{p,q}$ with all its homotopies.

We will sometimes depict evaluations of $\BF$ as follows:
\begin{equation}
\BF_{p,q}(\un{a_1},\dots,\un{a_p};m;\un{b_1},\dots,\un{b_q})(n) 
\hspace{2cm}
\begin{tikzpicture}[scale=1, baseline=-1ex]
\draw
    (0,0) 
 -- (0,.8)  -- (0,0)
 -- (-.866,.5)  -- (0,0)
 -- (-.5,.866)  -- (0,0)
 -- (0,1)  -- (0,0)
 -- (.5,.866)  -- (0,0) 
 -- (.866,.5)  -- (0,0) 
 -- (0,-.8)  -- (0,0)
 -- (-.866,-.5)  -- (0,0)
 -- (-.5,-.866)  -- (0,0)
 -- (0,-1)  -- (0,0)
 -- (.5,-.866)  -- (0,0) 
 -- (.866,-.5)  -- (0,0)  ;
 \draw [very thick] (-1,0) -- (1,0);
\filldraw[color=black, fill=white, very thick] (0,0) circle (.12);
\draw (1.2,-.6) node {$a_1$};
\draw (.7,-1.1) node {$a_2$};
\draw (-.4,-1.2) node {$\dots$};
\draw (-1.2,-.6) node {$a_p$};
\draw (-1.3,0) node {$m$};
\draw (-1.2,.6) node {$b_1$};
\draw (-.7,1.1) node {$b_2$};
\draw (.4,1.2) node {$\dots$};
\draw (1.2,.6) node {$b_q$};
\draw (1.3,0) node {$n$};
\end{tikzpicture}
\end{equation}

Note, that a homotopy inner product $\BF=\{\BF_{p,q}\}_{p,q}$ with $\BF_{p,q}=0$ for $p+q>0$ is precisely an inner product, i.e., a dg-bimodule map $\BF_{0,0}:\BM\to \BM^*$.

Let $\BF$ be a homotopy inner product for $\BM$. We will always assume that $\BF$ vanishes when any of the algebra inputs is the unit $\un{1}\in \un{A}$. Then, there is an induced map $\CF:=CH(\BF):\NCH^\bu(A,\BM)\to \NCH^\bu(A,\BM^*)$ by setting $\CF=CH(\BF)=\sum_{p,q\geq 0}CH(\BF)_{p,q}$, where $CH(\BF)_{p,q}:Hom(\un{A}^{\otimes r},\BM)\to Hom(\un{A}^{\otimes p+r+q},\BM^*)$ is
\begin{multline}\label{EQU:CH(F)}
(CH(\BF)_{p,q}(\varphi))(\un{a_1},\dots,\un{a_{p+r+q}})\\
:=(-1)^{|\varphi|\cdot(|\un{a_1}|+\dots+|\un{a_p}|)}\cdot \BF_{p,q}(\un{a_1}, \dots, \un{a_p};\varphi( \un{a_{p+1}}, \dots,\un{a_{p+r}});\un{a_{p+r+1}},\dots,\un{a_{p+r+q}}).
\end{multline}
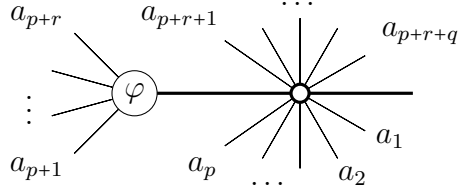
\[
\begin{tikzpicture}[scale=1, baseline=-1ex]
\draw
    (0,0) 
 -- (0,.8)  -- (0,0)
 -- (-1,.7)  -- (0,0)
 -- (-.5,.866)  -- (0,0)
 -- (0,1)  -- (0,0)
 -- (.5,.866)  -- (0,0) 
 -- (.866,.5)  -- (0,0) 
 -- (0,-.8)  -- (0,0)
 -- (-1,-.7)  -- (0,0)
 -- (-.5,-.866)  -- (0,0)
 -- (0,-1)  -- (0,0)
 -- (.5,-.866)  -- (0,0) 
 -- (.866,-.5)  -- (0,0)  ;
 \draw [very thick] (-2,0) -- (1.5,0);
\filldraw[color=black, fill=white, very thick] (0,0) circle (.12);
\draw (1.2,-.6) node {$a_1$};
\draw (.7,-1.1) node {$a_2$};
\draw (-.4,-1.2) node {$\dots$};
\draw (-1.3,-1) node {$a_p$};
\draw (-1.6,1) node {$a_{p+r+1}$};
\draw (0,1.2) node {$\dots$};
\draw (1.6,.8) node {$a_{p+r+q}$};
\draw
    (-2.2,0) 
 -- (-3,.8)  -- (-2.2,0)
 -- (-3.3,.3)  -- (-2.2,0)
 -- (-3.3,-.3)  -- (-2.2,0)
 -- (-3,-.8)  -- (-2.2,0);
\filldraw[color=black, fill=white] (-2.2,0) circle (.3); 
\draw (-2.2,0) node {$\varphi$};
\draw (-3.5,-1) node {$a_{p+1}$};
\draw (-3.5,1) node {$a_{p+r}$};
\draw (-3.6,-.1) node {$\vdots$};
\end{tikzpicture}
\]
A direct but lengthy calculation shows that 
\begin{equation}\label{EQU:delta(F)sharp}
CH(\delt \BF) (\varphi) = \delt\circ CH(\BF) (\varphi)-(-1)^{|\BF|}\cdot CH(\BF)\circ \delt(\varphi),\quad\quad \forall \varphi\in CH^\bu(A,\BM).
\end{equation}
\begin{cor}
If $\BF$ is a homotopy inner product, i.e., $\delt\BF=0$, then $\CF=CH(\BF):\NCH^\bu(A,\BM)\to \NCH^\bu(A,\BM^*)$ is a cochain map. By abuse of notation, we often write $\CF$ for the induced map on Hochschild cohomology. If $\BF=\delt\BF'$ for some $\BF'$, then the induced map on Hochschild cohomology vanishes, i.e., $HH(\BF)=0$.
\end{cor}

\subsection{Pullback under a dg-morphism}
Let $A$ and $B$ be two unital dg-algebras, and let $f:B\to A$ be a dg-algebra map (which is necessarily of degree $0$). Then any dg-bimodule $\BM=\BM\ov{A}$ over $A$ induces a dg-bimodule $\BM\ov{B}:=\BM$ over $B$ via the module structure $b.m:=f(b).m$ and $m.b:=m.f(b)$ for any $m\in \BM, b\in B$. Moreover, any dg-bimodule map $\BF\ov{A}:\BM\ov{A}\to \B{N}\ov{A}$ also induces a dg-bimodule map $\BF\ov{B}:\BM\ov{B}\to \B{N}\ov{B}$. Moreover, a homotopy inner product $\BF\ov{A}$ for $\BM\ov{A}$ also induces a homotopy inner product for $\BM\ov{B}$ via
\begin{equation*}
(\BF\ov{B})_{p,q}(\un{b_1},\dots, \un{b_p}; m;\un{c_1}, \dots, \un{c_q})(n) :=
(\BF\ov{A})_{p,q}(\un{f(b_1)},\dots, \un{f(b_p)}; m;\un{f(c_1)}, \dots, \un{f(c_q)})(n) 
\end{equation*}
for all $b_1,\dots, b_p,c_1,\dots, c_q,\in B$ and $m,n\in \BM\ov{B}$.

Moreover, for a dg-algebra map $f:B\to A$, we have the dg-modules $\B{B}\ov{B}$ and $\B{B}^*\ov{B}$, and, from $\B{A}\ov{A}$ and $\B{A}^*\ov{A}$, we also get $\B{A}\ov{B}$ and $\B{A}^*\ov{B}$. The dg-algebra map $f$ then induces a dg-bimodule map $\B{f}\ov{B}:\B{B}\ov{B}\to \B{A}\ov{B}$ and, by dualizing, the dg-bimodule map $\B{f}^*\ov{B}:\B{A}^*\ov{B}\to \B{B}^*\ov{B}$.

Combining the last two paragraphs, assume that $f:B\to A$ is a dg-algebra map, and that $\BF\ov{A}:\B{A}\ov{A}\to \B{A}^*\ov{A}$ is a homotopy inner product for $\B{A}$. Then, we get a transferred homotopy inner product $f(\BF)\ov{B}:\B{B}\ov{B}\to \B{B}^*\ov{B}$ for $\B{B}$, given by $f(\BF)\ov{B}:=\B{f}^*\ov{B}\circ \BF\ov{B}\circ \B{f}\ov{B}$,
\begin{equation*}
\xymatrix{ 
\B{B}\ov{B} \ar^{\B{f}\ov{B}}[rr] \ar_{f(\BF)\ov{B}}[d] && \B{A}\ov{B} \ar^{\BF\ov{B}}[d]  
\\ 
\B{B}^*\ov{B}  && \B{A}^*\ov{B}  \ar_{\B{f}^*\ov{B}}[ll] 
}
\end{equation*}
\begin{multline}\label{EQU:f(F)/B-as-F/A}
(f(\BF)\ov{B})_{p,q}(\un{b_1},\dots, \un{b_p}; m;\un{c_1}, \dots, \un{c_q})(n)
\\ =
(\BF\ov{A})_{p,q}(\un{f(b_1)},\dots, \un{f(b_p)}; f(m);\un{f(c_1)}, \dots, \un{f(c_q)})(f(n)),
\end{multline}
where $b_1,\dots, b_p,m,c_1,\dots, c_q,n\in B$.

Finally, for a dg-algebra map $f:B\to A$, and a dg-bimodule $\BM\ov{A}$, there is an induced cochain map on Hochschild cochains, $CH(f,\BM):\NCH^\bu(A,\BM\ov{A})\to\NCH^\bu(B,\BM\ov{B})$,
\begin{equation}\label{EQU:CH(f)-Def}
CH(f;\BM)(\varphi)(\un{b_1},\dots,\un{b_r}):=\varphi(\un{f(b_1)},\dots,\un{f(b_r)}), \quad\quad \forall \varphi\in\NCH^\bu(A,\BM\ov{A}), b_1,\dots, b_r\in B.
\end{equation}
With this we get the following lemma.
\begin{lem}
If $f:B\to A$ is a dg-algebra map, and $\BF\ov{A}:\B{A}\ov{A}\to \B{A}^*\ov{A}$ is a homotopy inner product for $\B{A}\ov{A}$, then the following diagram commutes.
\begin{equation}\label{EQU:commutative-Hoch-for-F}
\xymatrix{ 
\NCH^\bu(B,\B{B}\ov{B}) \ar^{CH(\B{f}\ov{B})}[rr] \ar_{CH(f(\BF)\ov{B})}[d] && \NCH^\bu(B,\B{A}\ov{B}) \ar^{CH(\BF\ov{B})}[d]  && \NCH^\bu(A,\B{A}\ov{A}) \ar_{CH(f,\B{A})}[ll] \ar^{CH(\BF\ov{A})}[d] 
\\ 
\NCH^\bu(B,\B{B}^*\ov{B})   && \NCH^\bu(B,\B{A}^*\ov{B})  \ar_{CH(\B{f}^*\ov{B})}[ll] && \NCH^\bu(A,\B{A}^*\ov{A}) \ar_{CH(f,\B{A}^*)}[ll]
}
\end{equation}
\end{lem}

\subsection{Isomorphism on Hochschild cohomology}
A particular case of interest is when a homotopy inner product $\BF:\B{A}\to \B{A}^*$ induces an isomorphism on Hochschild cohomology $\CF:=HH(\BF):HH^\bu(A,\B{A})\to HH^\bu(A,\B{A}^*)$. In this case, we can transfer any structure between these Hochschild cohomologies, in particular we can transfer the $\CB$ operator as it was done in theorem \ref{THM:main-theorem}.

Note that in equation \eqref{EQU:commutative-Hoch-for-F}, if the four horizontal maps are quasi-isomorphisms, then, obviously, the right vertical map $CH(\BF\ov{A})$ is a quasi-isomorphism iff the left vertical map $CH(f(\BF)\ov{B})$ is a quasi-isomorphism. 

\section{BV algebra on Hochschild cohomology}\label{SEC:BV-on-Hochschild}

We now review how homotopy inner products induce a BV algebra on Hochschild cohomology; see theorem \ref{THM:BV-on-Hoch}. Our main reference for this is \cite{T2}. We start by defining two operators $\CB$ and $\Delta^\BF$.

\begin{defn}
Consider a unital dg-algebra $A$.
\begin{enumerate}
\item We define Connes' $\CB$-operator (or more precisely the dual of Connes' $\CB$-operator) to be $\CB:\NCH^\bu(A,\B{A}^*)\to \NCH^\bu(A,\B{A}^*)$ given for $\varphi\in Hom(\un{A}^{\otimes r}, \B{A}^*)$ by $\CB(\varphi)\in Hom(\un{A}^{\otimes r-1}, \B{A}^*)$ with
\begin{multline}\label{EQU:B-operator-Def}
(\CB(\varphi))(\un{a_1},\dots, \un{a_{r-1}})(a_r)\\
:= \sum_{j=1}^r (-1)^{(\eps{j}{r})\cdot (\eps{1}{j-1})+|\un{a_r}|}
\cdot \varphi(\un{a_j},\dots, \un{a_r},\un{a_1},\dots, \un{a_{j-1}})(1),
\end{multline}
A direct but lengthy computation shows that $\delt\circ \CB(\varphi)=-\CB\circ \delt(\varphi)$. By abuse of notation, we denote the induced map on Hochschild cohomology by $\CB:HH^\bu(A,\B{A}^*)\to HH^\bu(A,\B{A}^*)$ as well.
\item
Next, assume that we also have a homotopy inner product $\BF=\{\BF_{p,q}:\un{A}^{\ot p}\ot \B{A}\ot \un{A}^{\ot q}\to \B{A}^*\}_{p,q\geq 0}$. Then, we define the operator $\CZ^\BF:=\sum\limits_{p,q\geq 0}\CZ^\BF_{p,q}:\NCH^\bu(A,\B{A})\to \NCH^\bu(A,\B{A}^*)$, where, for $\varphi\in Hom(\un{A}^{\otimes r}, \B{A})$, we set $\CZ^\BF_{p,q}(\varphi)\in Hom(\un{A}^{\otimes p+r+q-1},\B{A}^*)$ to be given by
\begin{align}\label{EQU:DEF-of-ZF}
&(\CZ^\BF_{p,q}(\varphi))(\un{a_1},\dots, \un{a_{p+r+q-1}})(a_{p+r+q})
\\
\nonumber &:=\sum_{j=p+1}^{p+q} (-1)^{|\BF|+(|\varphi|+1)\cdot (|\un{a_1}|+\dots+|\un{a_{j-1}}|)+1}
\\
\nonumber & \quad\quad \cdot \BF_{p,q}\Big(\un{a_1},\dots, \un{a_{p}};1; \un{a_{p+1}},\dots,\un{a_{j-1}},\un{\varphi(\un{a_{j}},\dots, \un{a_{j+r-1}})},\un{a_{j+r}},\dots,\un{a_{p+r+q-1}}\Big)\Big(a_{p+r+q}\Big)
\\
\nonumber &\quad+\sum_{j=1}^r (-1)^{(\eps{j}{p+r+q})\cdot (\eps{1}{j-1})+|\un{a_{p+r+q}}|+|\varphi| \cdot (|\un{a_{j}}|+\dots+|\un{a_{j+p+q-1}}|)}
\\
\nonumber & \quad\quad \cdot   \BF_{p,q}\Big(\un{a_{j}},\dots,\un{a_{j+p-1}};1; \un{a_{j+p}},\dots,\un{a_{j+p+q-1}}\Big)\Big(\varphi(\un{a_{p+j+q}},\dots, \un{a_{p+r+q}},\un{a_1},\dots, \un{a_{j-1}})\Big)
\\
\nonumber &\quad +\sum_{j=1}^{p} (-1)^{|\BF|+(|\varphi|+1)\cdot (\eps{1}{j-1})+1} 
\\
\nonumber & \quad\quad
\cdot\BF_{p,q}\Big(\un{a_1},\dots,\un{a_{j-1}},\un{\varphi(\un{a_{j}},\dots, \un{a_{j+r-1}})},\un{a_{j+r}},\dots, \un{a_{p+r-1}};1;\un{a_{p+r}},\dots,\un{a_{p+r+q-1}} \Big)\Big(a_{p+r+q}\Big)
\end{align}
One can check again that $\delt\circ \CZ^\BF(\varphi)=-(-1)^{|\BF|}\cdot \CZ^\BF\circ \delt(\varphi)$. By abuse of notation, we denote the induced map on Hochschild cohomology by $\CZ^\BF:HH^\bu(A,\B{A})\to HH^\bu(A,\B{A}^*)$ as well.

We remark that $\CZ^\BF$ has appeared in \cite[lemma 17]{T2} as the operation associated to the symbol
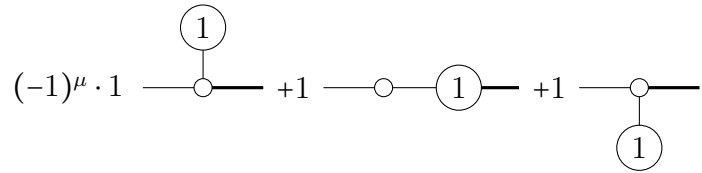
\[
\begin{tikzpicture}[scale=1, baseline=0ex]
\draw   (-2.4,.8)  -- (-2.4,0) -- (-3.2,0);
\draw [very thick]  (-1.6,0)  -- (-2.4,0);
\filldraw[color=black, fill=white] (-2.4,0) circle (.12);
\filldraw[color=black, fill=white] (-2.4,.8) circle (.3); 
\draw (-4.2,0) node {$(-1)^\mu\cdot 1$};
\draw (-2.4,.8) node {$1$};
\draw   (-.8,0)  -- (.8,0);
\draw [very thick]  (1,0)  -- (1.8,0);
\filldraw[color=black, fill=white] (0,0) circle (.12);
\filldraw[color=black, fill=white] (1,0) circle (.3); 
\draw (1,0) node {$1$};
\draw (-1.2,0) node {$+1$};
\draw (2.2,0) node {$+1$};
\draw   (2.6,0)  -- (3.4,0) -- (3.4,-.8);
\draw [very thick]  (3.4,0)  -- (4.2,0);
\filldraw[color=black, fill=white] (3.4,0) circle (.12);
\filldraw[color=black, fill=white] (3.4,-.8) circle (.3); 
\draw (3.4,-.8) node {$1$};
\end{tikzpicture}
\]
\item
Let $\BF$ be a homotopy inner product for $\B{A}$. Assume, moreover, that $\BF:\B{A}\to\B{A}^*$ induces an isomorphism on Hochschild cohomology, $\CF:=HH(\BF):HH^\bu(A,\B{A})\to HH^\bu(A,\B{A}^*)$. Then denote by $\De^\BF:HH^\bu(A,\B{A})\to HH^\bu(A,\B{A})$ the composition $\De^\BF:=\CF^{-1}\circ \CZ^\BF$,
\[
HH^\bu(A,\B{A})\stackrel{\CZ^\BF}\longrightarrow HH^\bu(A,\B{A}^*)\stackrel{\CF^{-1}}\longrightarrow HH^\bu(A,\B{A})
\]
\end{enumerate}
\end{defn}

\begin{lem}\label{LEM:BV-Delta-conditions}
Let $A$ be a unital dg-algebra.
\begin{enumerate}
\item\label{ITEM:B2=0} On Hochschild cochains $HH^\bu(A,\B{A}^*)$, we have that $\CB^2=0$.
\item\label{ITEM:Delta-F-deviation} Assume that $\BF:\B{A}\to\B{A}^*$ is a homotopy inner product which induces an isomorphism on Hochschild cohomology $\CF:HH^\bu(A,\B{A})\to HH^\bu(A,\B{A}^*)$. Then the deviation of $\Delta^\BF$ from being a derivation of the cup product  is the usual Gerstenhaber bracket.

Here the cup product on Hochschild cohomology is given for Hochschild cochains $\varphi\in Hom(\un{A}^{\otimes r},\B{A})$ and $\rho\in Hom(\un{A}^{\otimes s},\B{A})$ to be $\varphi\smile \rho\in Hom(\un{A}^{\otimes {r+s}},\B{A})$ with
\begin{equation}\label{EQU:cup-product}
(\varphi\smile \rho)(\un{a_1},\dots, \un{a_{r+s}}):=(-1)^{|\rho|\cdot(|\un{a_1}|+\dots+ |\un{a_{r}}|)}\cdot \varphi(\un{a_1},\dots, \un{a_{r}})\cdot \rho(\un{a_{r+1}},\dots, \un{a_{r+s}}).
\end{equation}
\end{enumerate}
\end{lem}
\begin{proof}
For \eqref{ITEM:B2=0}, note that $\CB^2=0$ follows since normalized Hochschild cochains vanish when any input is the unit $\un{1}$. Part \eqref{ITEM:Delta-F-deviation} was proved in \cite[section 3.3]{T2}.
\end{proof}

Since lemma \ref{LEM:BV-Delta-conditions} \eqref{ITEM:B2=0} and \eqref{ITEM:Delta-F-deviation} are conditions needed for a BV algebra, i.e., a square zero operator (here: $\CB$) whose deviation from being a derivation is a Gerstenhaber bracket (here: $\De^\BF$), we make the following definition.
\begin{defn}\label{DEF:PD-structure}
Let $A$ be a dg-algebra, and let $\BF:\B{A}\to \B{A}^*$ be a homotopy inner product for $\B{A}$. We call $\BF$ a Poincar\'e duality structure for $A$, if it satisfies if it satisfies the following two conditions:
\begin{enumerate}
\item\label{ITEM:Fsharp-HH-iso}
$\BF$ induces an isomorphism of graded modules on Hochschild cohomology $\CF=HH(\BF):HH^\bu(A,\B{A})\stackrel{\cong}\longrightarrow HH^\bu(A,\B{A}^*)$.
\item\label{ITEM:B=Delta}
Transferring $\CB$ from $HH^\bu(A,\B{A}^*)$ to $HH^\bu(A,\B{A})$ via $\CF$ equals $\De^\BF$, i.e.,
\[
\xymatrix{ 
HH^\bu(A,\B{A}) \ar@/_2.0pc/@[][d]_{\CF} 
\\
HH^\bu(A,\B{A}^*)  \ar@/_2.0pc/@[][u]_{\CF^{-1}} \ar@{->}@(lu,ru)^{\CB}
}
\hspace{6mm}
\raisebox{-7mm}{$\CF^{-1}\circ \CB \circ \CF=\De^\BF=\CF^{-1}\circ \CZ^\BF$}
\hspace{6mm}
\xymatrix{ 
HH^\bu(A,\B{A}) \ar@{->}@(ld,rd)_{\De^\BF}
\\
\quad
}
\raisebox{-7mm}{$=$}
\xymatrix{ 
HH^\bu(A,\B{A}) \ar@/_2.0pc/@[][d]^{\CZ^\BF} 
\\
HH^\bu(A,\B{A}^*)  \ar@/_2.0pc/@[][u]^{\CF^{-1}}
}
\]
\end{enumerate}
\end{defn}

With this, we obtain the following theorem, which is \cite[theorem 2]{T2}.
\begin{thm}\label{THM:BV-on-Hoch}
Let $A$ be a dg-algebra, and let $\BF:\B{A}\to \B{A}^*$ be a Poincar\'e duality structure for $A$.
Then, $HH^\bu(A,\B{A})$ together with the cup product and $\De^\BF$ is a BV algebra.
\end{thm}

\begin{exs}
We now give a few examples for how one can check whether a homotopy inner product is a Poincar\'e duality structure.
\begin{enumerate}
\item\label{ITME:IP-are-PD}
Let $\BF$ be an inner product for a dg-algebra $A$ (i.e., the only non-zero component is $\BF_{0,0}$), and assume that $\CF$ induces an isomorphism on Hochschild cohomology. Then $\BF$ is a Poincar\'e duality structure.

(To see this, we note that this is a special case of the next example \eqref{ITEM: weakly-symmetric-are-PD} below, since the condition $(\delt_1 \BF)_{0,1}=0$ gives (using \eqref{EQU:de1(F)}) that $\BF(mb)(n)=\BF(m)(bn)$, and $(\delt_1 \BF)_{1,0}=0$ gives $\BF(am)(n)=(-1)^{|a|\cdot (|m|+|n|)}\BF(m)(na)$ for all $m,n,a,b\in A$, and, thus $\BF(m)(n)=\BF(1\cdot m)(n)=\BF(1)(m\cdot n)=(-1)^{|m|\cdot |n|}\BF(n\cdot 1)(m)=(-1)^{|m|\cdot |n|}\BF(n)(m)$.)
\item\label{ITEM: weakly-symmetric-are-PD}
Let $\BF$ be a homotopy inner product for a dg-algebra $A$, and assume that $\CF$ induces an isomorphism on Hochschild cohomology. Assume further that $\BF$ is invariant under cyclic rotation of the first $p+1$ and last $q+1$ inputs, i.e.,  $\forall p,q\geq 0$, $\forall a_1,\dots, a_p,m,b_1, \dots, b_q,n\in A$:
\begin{align*}
\BF_{p,q}(\un{a_1},\dots, \un{a_p}; m;\un{b_1}, \dots, \un{b_q})(n)
&=(-1)^\ep\cdot
\BF_{q,p}(\un{b_1}, \dots, \un{b_q};n;\un{a_1},\dots, \un{a_p})(m)
\\
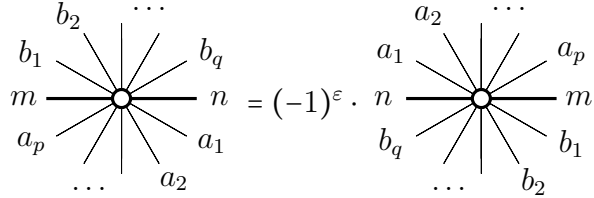
\begin{tikzpicture}[scale=1, baseline=-1ex]
\draw
    (0,0) 
 -- (0,.8)  -- (0,0)
 -- (-.866,.5)  -- (0,0)
 -- (-.5,.866)  -- (0,0)
 -- (0,1)  -- (0,0)
 -- (.5,.866)  -- (0,0) 
 -- (.866,.5)  -- (0,0) 
 -- (0,-.8)  -- (0,0)
 -- (-.866,-.5)  -- (0,0)
 -- (-.5,-.866)  -- (0,0)
 -- (0,-1)  -- (0,0)
 -- (.5,-.866)  -- (0,0) 
 -- (.866,-.5)  -- (0,0)  ;
 \draw [very thick] (-1,0) -- (1,0);
\filldraw[color=black, fill=white, very thick] (0,0) circle (.12);
\draw (1.2,-.6) node {$a_1$};
\draw (.7,-1.1) node {$a_2$};
\draw (-.4,-1.2) node {$\dots$};
\draw (-1.2,-.6) node {$a_p$};
\draw (-1.3,0) node {$m$};
\draw (-1.2,.6) node {$b_1$};
\draw (-.7,1.1) node {$b_2$};
\draw (.4,1.2) node {$\dots$};
\draw (1.2,.6) node {$b_q$};
\draw (1.3,0) node {$n$};
\end{tikzpicture}
&=(-1)^\ep\cdot
\begin{tikzpicture}[scale=1, baseline=-1ex]
\draw
    (0,0) 
 -- (0,.8)  -- (0,0)
 -- (-.866,.5)  -- (0,0)
 -- (-.5,.866)  -- (0,0)
 -- (0,1)  -- (0,0)
 -- (.5,.866)  -- (0,0) 
 -- (.866,.5)  -- (0,0) 
 -- (0,-.8)  -- (0,0)
 -- (-.866,-.5)  -- (0,0)
 -- (-.5,-.866)  -- (0,0)
 -- (0,-1)  -- (0,0)
 -- (.5,-.866)  -- (0,0) 
 -- (.866,-.5)  -- (0,0)  ;
 \draw [very thick] (-1,0) -- (1,0);
\filldraw[color=black, fill=white, very thick] (0,0) circle (.12);
\draw (1.2,-.6) node {$b_1$};
\draw (.7,-1.1) node {$b_2$};
\draw (-.4,-1.2) node {$\dots$};
\draw (-1.2,-.6) node {$b_q$};
\draw (-1.3,0) node {$n$};
\draw (-1.2,.6) node {$a_1$};
\draw (-.7,1.1) node {$a_2$};
\draw (.4,1.2) node {$\dots$};
\draw (1.2,.6) node {$a_p$};
\draw (1.3,0) node {$m$};
\end{tikzpicture}
\end{align*}
where $(-1)^\ep=(-1)^{(|\un{a_1}|+\dots+ |\un{a_p}|+|m|)\cdot (|\un{b_1}|+ \dots+ |\un{b_q}|+|n|)}$.
Then $\BF$ is a Poincar\'e duality structure. This was proved in \cite[lemma 17]{T2}.
\item\label{ITEM:non-PD-example}
In the examples in section \ref{SEC:compute-Hochschild-BV} of this paper we will check if a given homotopy inner product $\BF$ is a Poincar\'e duality structure or not by explicitly computing $\CF$, $\CB$ and $\CZ^\BF$; cf. the examples given in \ref{SUBSEC:Delta-F-strict}, \ref{SUBSEC:Delta-F-homotopies} and \ref{SUBSEC:PD-counterexample}.

In particular, it is worth noting that there exist homotopy inner products $\BF$ for which $\CF$ is an isomorphism, but for which $\CF^{-1}\circ \CB \circ \CF\neq \Delta^\BF$. We provide an explicit example for such a homotopy inner product in \ref{SUBSEC:PD-counterexample}.
\end{enumerate}
\end{exs}

\section{Computations of bimodule maps with higher homotopies}\label{SEC:compute-bimodule-maps}

In this section, we will give explicit bimodule maps and bimodule maps with higher homotopies for specific dg-algebras. In particular we will compute the homotopy inner product on $H^\bu(S^2;\Z_2)$ coming from a pullback of a local homotopy inner product. In the remainder of this paper --with the exception of observation \ref{OBS:not-bimodule}-- the ground ring will always be $R=\Z_2$.

\begin{ex}[$H^\bu(S^2;\Z_2)$ inner product without homotopies]\label{EX:HS2Z2-no-homotopies}
Consider the dg-algebra $A:=H^\bu(S^2;\Z_2)\cong \Z_2.e\oplus \Z_2.s$ with zero differential and degrees $|e|=0$ and $|s|=-2$, and where $e$ is the unit and $s\cdot s=0$. For the dg module and dual dg module of $A$ we use the notation $\B{A}\cong \Z_2.e\oplus \Z_2.s$ and $\B{A}^*\cong \Z_2.e^*\oplus \Z_2.s^*$, respectively, where $e^*$ and $s^*$ are the duals of $e$ and $s$ with $|e^*|=0$ and $|s^*|=2$ and the module structure is given by $e.e^*=e^*.e=e^*, e.s^*=s^*.e=s^*, s.e^*=e^*.s=0, s.s^*=s^*.s=e^*$ (see \ref{SUBSEC:basic-setup}).

Define the dg bimodule map $\BF:\B{A}\to \B{A}^*$ by $\BF(s)=e^*$ and $\BF(e)=s^*$,
\[
\begin{tikzcd}[column sep=small]
\B{A}\arrow{d}{\BF} &=\dots & \gen{s} \arrow{l}\arrow{drr}{\BF} & 0 \arrow{l} &\gen{e}\arrow{l}\arrow{l}\arrow{drr}{\BF} &0\arrow{l} &0\arrow{l} & \dots \arrow{l} \\
\B{A}^*&=\dots & 0\arrow{l} & 0\arrow{l} &\gen{e^*}\arrow{l} &0\arrow{l} &\gen{s^*}\arrow{l} & \dots \arrow{l}
\end{tikzcd}
\]
Thus, $\BF:\B{A}\to \B{A}^*$ is the map $\BF(s)(e)=1, \BF(s)(s)=0, \BF(e)(e)=0, \BF(e)(s)=1$ (in other words, $\BF$ is given by capping with $s^*$). One can check directly that $\BF$ is a dg bimodule map. Note, that when interpreting $\BF$ as an inner product $<,>=\BF:\B{A}\otimes\B{A}\to \Z_2$,  the only non-vanishing inner products are $<s,e>=1$ and $<e,s>=1$. 
\end{ex}

\begin{ex}[$H^\bu(S^2;\Z_2)$ inner product with homotopies]\label{EX:HS2Z2-with-homotopies}\label{EX:HS2Z2-with-homotopies} We next consider the same algebra $A$ and module structures $\B{A}$ and $\B{A}^*$ as in example \ref{EX:HS2Z2-no-homotopies}, but we define an inner product $\wt\BF$ for $\B{A}$ with higher homotopies (in the sense of \ref{SUBSEC:IP-and-HIP}). In fact, $\wt \BF$ has its only non-zero components given by $\wt \BF_{0,0}:\B{A}\to \B{A}^*$ and $\wt \BF_{2,0}:\un{A}\ot \un{A}\ot\B{A}\to \B{A}^*$ (see \ref{SUBSEC:IP-and-HIP}) via
\begin{align}\label{EQU:HIP-in-H(S2,Z2)}
&\wt \BF_{0,0}(s)(e)=1, \quad\quad\quad \wt \BF_{0,0}(e)(s)=1,\quad\quad\quad \wt \BF_{2,0}(\un{s},\un{s};e)(e)=1.
\\ & \nonumber
\begin{tikzpicture}[scale=.7, baseline=-1ex]
  \draw [very thick] (-1,0) -- (1,0);
\filldraw[color=black, fill=white, very thick] (0,0) circle (.12);
\draw (-1.3,0) node {$s$};
\draw (1.3,0) node {$e$};
\end{tikzpicture}
\hspace{16.5mm}
\begin{tikzpicture}[scale=.7, baseline=-1ex]
  \draw [very thick] (-1,0) -- (1,0);
\filldraw[color=black, fill=white, very thick] (0,0) circle (.12);
\draw (-1.3,0) node {$e$};
\draw (1.3,0) node {$s$};
\end{tikzpicture}
\hspace{17mm}
\begin{tikzpicture}[scale=.7, baseline=-1ex]
\draw
    (0,0) 
 -- (-.5,-.866)  -- (0,0)
 -- (.5,-.866)  -- (0,0) ;
  \draw [very thick] (-1,0) -- (1,0);
\filldraw[color=black, fill=white, very thick] (0,0) circle (.12);
\draw (.7,-1.1) node {$s$};
\draw (-.7,-1.1) node {$s$};
\draw (-1.3,0) node {$e$};
\draw (1.3,0) node {$e$};
\end{tikzpicture}
\end{align}
It is again a direct check that $\wt \BF$ is a homotopy inner product, i.e., it satisfies all equation required by $0=\delt \BF=\delt_1 \BF$ given by \eqref{EQU:de1(F)} (since $d$ and thus $\delt_0$ vanishes). Explicitly, these equations are:
\begin{align*}
\wt \BF_{0,0}(a\cdot a_1)(\wt a)=&\wt \BF_{0,0}(a)(a_1\cdot \wt a), \\
\wt \BF_{0,0}(a_1\cdot a)(\wt a)=&\wt \BF_{0,0}(a)(\wt a\cdot a_1), \\
\wt \BF_{2,0}(\un{a_1},\un{a_2};a\cdot a_3)(\wt a)=&\wt \BF_{2,0}(\un{a_1},\un{a_2};a)(a_3\cdot \wt a),\\
0=&\wt \BF_{2,0}(\un{a_1},\un{a_2};a_3\cdot a)(\wt a)+
\wt \BF_{2,0}(\un{a_1},\un{a_2\cdot a_3}; a)(\wt a)\\
&+\wt \BF_{2,0}(\un{a_1\cdot a_2},\un{a_3}; a)(\wt a)+
\wt \BF_{2,0}(\un{a_2},\un{a_3}; a)(\wt a\cdot a_1)
\end{align*}
for all $a,\wt a\in \B{A}$, and $a_1, a_2,a_3\in A$. 
\end{ex}

We next give formulas for calculating homotopy inner products (over $\Z_2$) for (triangulated) $2$-dimensional spaces on the cochain level. The following observation notes that higher homotopies naturally appear for inner products on the cochain level.
\begin{obs} \label{OBS:not-bimodule}
Let $A$ be a unital dg-algebra over any commutative ring $R$, so that both $\B{A}:=A$ and $\B{A}^*$ are dg-modules over $A$, (see \ref{SUBSEC:basic-setup}). Now, let $x\in \B{A}^*$ be a fixed closed element, $d(x)=0$. Define $\BF:\B{A}\to \B{A}^*$ for $a\in \B{A}$ by setting $\BF(a)\in \B{A}^*$ evaluated on some $\wt a\in \B{A}$ to be $\BF(a)(\wt a):=(x\frown a)(\wt a)=x(a\cdot \wt a)$.

{\bf Claim:} $\BF$ is chain map, and a graded right module map. $\BF$ is in general \emph{not} a graded left module map. A sufficient condition for $\BF$ being a graded left module map is that $A$ is graded commutative.
\begin{proof}
First, $\BF$ is chain map, since for $a,\wt a\in \B{A}$ and $dx=0$, we have $\BF(da)(\wt a)=x(da\cdot \wt a)=x(d(a\cdot \wt a)-(-1)^{|a|}a\cdot d\wt a)=(-1)^{|x|+1}dx(a\cdot \wt a)-(-1)^{|a|}\cdot \BF(a)(d\wt a)=-(-1)^{|a|}\cdot (-1)^{|\BF(a)|+1}\cdot d(\BF(a))(\wt a)=(-1)^{|\BF|}\cdot d(\BF(a))(\wt a)$. Next, $\BF$ is a graded right module map, since for $a,\wt a\in \B{A}$, $a_1\in A$, we  have $(\BF(a). a_1)(\wt a)=\BF(a)(a_1\cdot \wt a)=x(a\cdot a_1\cdot \wt a)=\BF(a.a_1)(\wt a)$. To check when $\BF$ is a graded left module map, compute $\BF(a_1. a)(\wt a)=x(a_1\cdot a\cdot \wt a)$ and $(a_1. \BF(a))(\wt a)=(-1)^{|a_1|\cdot (|\BF(a)|+|\wt a|)}\cdot \BF(a)(\wt a\cdot a_1)=(-1)^{|a_1|\cdot (|\BF|+|a|+|\wt a|)}\cdot x(a\cdot \wt a\cdot a_1)$. Thus, $\BF$ is in general  \emph{not} a graded left module map. Moreover, if $a_1\cdot a\cdot \wt a=(-1)^{|a_1|\cdot (|a|+|\wt a|)}a\cdot \wt a\cdot a_1$ for all $a,\wt a,a_1$, then $(a_1. \BF(a))(\wt a)=(-1)^{|a_1|\cdot |\BF|}\cdot\BF(a_1.a)(\wt a)$, and thus $\BF$ will be a graded left module map.
\end{proof}

In particular, assume that $X$ is a closed, oriented manifold, and $R$ is a commutative ring. In order to calculate the BV algebra on $HH^\bu(C^\bu( X;R); (C^\bu( X;R))^*)$ one might try to use capping a cochain with a fundamental cycle $x$ of $X$ as the appropriate dg-bimodule map $\BF=x\frown -:C^\bu( X;R)\to C_\bu( X;R)\stackrel {incl}\hookrightarrow (C^\bu( X;R))^*$. However, the above observation shows that capping on the (co-)chain level is in general \emph{not} a dg-bimodule map, and thus does not induce a chain map on the Hochschild cochains. One way to resolve this issue is to provide higher homotopies for the left module structure, i.e., to provide a homotopy inner product, which then does give a corresponding chain map on Hochschild cochains. This is what is done in this paper for $S^2$ with $\Z_2$ coefficients.
\end{obs}

\begin{ex}[The $0$-simplex]\label{EX:F-for-the-0-simplex}
Let $A^{[0]}:=\Z_2.e_0$, where $e_0$ is the unit. A homotopy inner product $\BF^{[0]}:\B{A}^{[0]}\to (\B{A}^{[0]})^*$ is given by $\BF^{[0]}(e_0)(e_0)=1$.

We will also vary the superscript in an obvious way, i.e., the dg-algebra $A^{[1]}:=\Z_2.e_1$ has the homotopy inner product $\BF^{[1]}:\B{A}^{[1]}\to (\B{A}^{[1]})^*, \BF^{[1]}(e_1)(e_1)=1$, etc.
\end{ex}

\begin{ex}[The $1$-simplex]\label{EX:F-for-the-1-simplex}
Let $A^{[01]}:=\Z_2.e_0\oplus \Z_2.e_1\oplus \Z_2.b_{01}$, where $|e_0|=|e_1|=0$ and $|b_{01}|=-1$ with differential $d(e_0)=d(e_1)=b_{01}$. The product is the usual cup product of simplicial cochains, i.e., $e_0\cdot e_0=e_0$, $e_1\cdot e_1=e_1$, and $e_0\cdot b_{01}=b_{01}\cdot e_1=b_{01}$. Note that the unit of $A^{[01]}$ is $e_0+e_1$.

We want to define an inner product $\BF$ with lowest component non-vanishing only for $\BF(e_0)(b_{01})=\BF(b_{01})(e_1)=1$. Note that $\BF$ is not a chain map, but $\BF(d(a))(\wt a)+\BF(a)(d(\wt a))=\BF^{[0]}(a)(\wt a)+\BF^{[1]}(a)(\wt a)$ for all $a,\wt a \in A^{[01]}$. Moreover, $\BF$ is a graded right module map, i.e., $\BF(a\cdot a_1)(\wt a)=\BF(a)(a_1\cdot \wt a)$ for all $a, \wt a, a_1\in A^{[01]}$, but $\BF$ is not a graded left module map, since $\BF(e_0\cdot e_0)(b_{01})=1\neq 0 =\BF(e_0)(b_{01}\cdot e_0)$.

There is an inductive procedure (involving choices at each stage) for obtaining higher homotopies that provide $\BF$ with left modules homotopies, making it into a homotopy inner product up to $\BF^{[0]}$ and $\BF^{[1]}$ from example \ref{EX:F-for-the-0-simplex} (interpreted as maps $\B{A}^{[01]}\to (\B{A}^{[01]})^*$) . This procedure was described in \cite[Proposition 3.1.2]{TZS}. Performing the induction leads to a sequence of maps, which was first stated in \cite[Proposition B.2]{TZS}:
\begin{align}\label{EQU:Fk0-for-1-simplicies}
\forall k\geq 0: &
\quad\quad \BF^{[01]}_{k,0}(\un{b_{01}},\dots,\un{b_{01}};e_0)(b_{01})=1,
\quad\quad \BF^{[01]}_{k,0}(\un{b_{01}},\dots,\un{b_{01}};b_{01})(e_1)=1,\\
\nonumber &
\hspace{13mm}
\begin{tikzpicture}[scale=.7, baseline=-1ex]
\draw
    (0,0) 
 -- (-.866,-.5)  -- (0,0)
 -- (-.5,-.866)  -- (0,0)
 -- (0,-1)  -- (0,0)
 -- (.5,-.866)  -- (0,0) 
 -- (.866,-.5)  -- (0,0) ;
 \draw [very thick] (-1,0) -- (1,0);
\filldraw[color=black, fill=white, very thick] (0,0) circle (.12);
\draw (1.5,-.6) node {$b_{01}$};
\draw (.8,-1.2) node {$b_{01}$};
\draw (-.4,-1.2) node {$\dots$};
\draw (-1.3,-.6) node {$b_{01}$};
\draw (-1.4,0) node {$e_0$};
\draw (1.5,0) node {$b_{01}$};
\end{tikzpicture}
\hspace{31mm}
\begin{tikzpicture}[scale=.7, baseline=-1ex]
\draw
    (0,0) 
 -- (-.866,-.5)  -- (0,0)
 -- (-.5,-.866)  -- (0,0)
 -- (0,-1)  -- (0,0)
 -- (.5,-.866)  -- (0,0) 
 -- (.866,-.5)  -- (0,0) ;
 \draw [very thick] (-1,0) -- (1,0);
\filldraw[color=black, fill=white, very thick] (0,0) circle (.12);
\draw (1.5,-.6) node {$b_{01}$};
\draw (.8,-1.2) node {$b_{01}$};
\draw (-.4,-1.2) node {$\dots$};
\draw (-1.3,-.6) node {$b_{01}$};
\draw (-1.4,0) node {$b_{01}$};
\draw (1.5,0) node {$e_1$};
\end{tikzpicture}
\end{align}
All other inner products are zero. Note, this resolves the above problem of $\BF^{[01]}_{0,0}$ not being a graded left module map, since, now, $\BF^{[01]}_{0,0}$ is a graded left module map up to the homotopy $\BF^{[01]}_{1,0}$; for example:
\begin{multline*}
(\delt\BF^{[01]})_{1,0}(\un{e_0};e_0)(b_{01})
=\BF^{[01]}_{1,0}(\un{d(e_0)};e_0)(b_{01})+\BF^{[01]}_{1,0}(\un{e_0};d(e_0))(b_{01})
\\
+\BF^{[01]}_{0,0}(e_0\cdot e_0)(b_{01})+\BF^{[01]}_{0,0}(e_0)(b_{01}\cdot e_0)
=1+0+1+0=0.
\end{multline*}

From \cite[Proposition B.2]{TZS}, which we will also prove in appendix \ref{APP:Proof-of-HIP-on-2-simplex}, we have that:
\begin{equation}\label{EQU:D(F01)=F0+F1}
\delt \BF^{[01]}=\BF^{[0]}+\BF^{[1]}
\end{equation}

Again, we will need to vary the superscript in an obvious way, i.e., for the dg-algebra $A^{[12]}$ there are maps $\BF^{[12]}$ which are given by replacing $0$ and $1$ in the above with $1$ and $2$, respectively, etc.
\end{ex}

\begin{ex}[The $2$-simplex]\label{EX:F-for-the-2-simplex}
We now describe inner product maps for the $2$-simplex, which extends the previous two examples of the $0$- and $1$-simplex. Let $A^{[012]}:=\Z_2.e_0\oplus \Z_2.e_1\oplus \Z_2.e_2\oplus \Z_2.b_{01}\oplus \Z_2.b_{02}\oplus \Z_2.b_{12}\oplus \Z_2.c_{012}$ with $|e_j|=0$, $|b_{ij}|=-1$ and $|c_{012}|=-2$, and differential $d(e_0)=b_{01}+b_{02}$, $d(e_1)=b_{01}+b_{12}$, $d(e_2)=b_{02}+b_{12}$ and $d(b_{ij})=c_{012}$ for all $0\leq i<j\leq 2$. The multiplication is non-zero only for
\begin{align}\label{EQU:Product-of-2-simplex}
&\forall j:\quad
e_j\cdot e_j=e_j, \quad
\forall i<j:\quad
e_i\cdot b_{ij}=b_{ij}, \quad
b_{ij}\cdot e_j=b_{ij}, \\ \nonumber
&\quad\quad e_0\cdot c_{012}=c_{012}, \quad
c_{012}\cdot e_2=c_{012}, \quad
b_{01}\cdot b_{12}=c_{012}.
\end{align}

Now, following the procedure (which uses locality) from \cite[Proposition 3.1.2]{TZS}, we define maps $\BF^{[012]}_{k,0}$  whose only non-zero maps are given by the following equations \eqref{EQU:Fk0-02-2-012}-\eqref{EQU:F-c-c}:
\begin{align}\label{EQU:Fk0-02-2-012}
\forall k\geq 0:\quad&
\BF^{[012]}_{k,0}(\un{b_{02}},\dots,\un{b_{02}};c_{012})(e_2)=1, 
\\
&\BF^{[012]}_{k,0}(\un{b_{02}},\dots,\un{b_{02}};e_{0})(c_{012})=1,
\\
&\BF^{[012]}_{k,0}(\un{b_{02}},\dots,\un{b_{02}};b_{01})(b_{12})=1,
\\
\nonumber
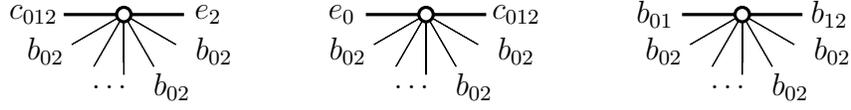
\begin{tikzpicture}[scale=.8, baseline=-1ex]
\draw
    (0,0) 
 -- (0,-.8)  -- (0,0)
 -- (-.866,-.5)  -- (0,0)
 -- (-.5,-.866)  -- (0,0)
 -- (0,-1)  -- (0,0)
 -- (.5,-.866)  -- (0,0) 
 -- (.866,-.5)  -- (0,0) ;
 \draw [very thick] (-1,0) -- (1,0);
\filldraw[color=black, fill=white, very thick] (0,0) circle (.12);
\draw (1.5,-.6) node {$b_{02}$};
\draw (.8,-1.2) node {$b_{02}$};
\draw (-.2,-1.2) node {$\dots$};
\draw (-1.3,-.6) node {$b_{02}$};
\draw (1.4,0) node {$e_2$};
\draw (-1.5,0) node {$c_{012}$};
\end{tikzpicture}
& \hspace{1cm}
\begin{tikzpicture}[scale=.8, baseline=-1ex]
\draw
    (0,0) 
 -- (0,-.8)  -- (0,0)
 -- (-.866,-.5)  -- (0,0)
 -- (-.5,-.866)  -- (0,0)
 -- (0,-1)  -- (0,0)
 -- (.5,-.866)  -- (0,0) 
 -- (.866,-.5)  -- (0,0) ;
 \draw [very thick] (-1,0) -- (1,0);
\filldraw[color=black, fill=white, very thick] (0,0) circle (.12);
\draw (1.5,-.6) node {$b_{02}$};
\draw (.8,-1.2) node {$b_{02}$};
\draw (-.2,-1.2) node {$\dots$};
\draw (-1.3,-.6) node {$b_{02}$};
\draw (1.5,0) node {$c_{012}$};
\draw (-1.4,0) node {$e_{0}$};
\end{tikzpicture}
\hspace{1cm}
\begin{tikzpicture}[scale=.8, baseline=-1ex]
\draw
    (0,0) 
 -- (0,-.8)  -- (0,0)
 -- (-.866,-.5)  -- (0,0)
 -- (-.5,-.866)  -- (0,0)
 -- (0,-1)  -- (0,0)
 -- (.5,-.866)  -- (0,0) 
 -- (.866,-.5)  -- (0,0) ;
 \draw [very thick] (-1,0) -- (1,0);
\filldraw[color=black, fill=white, very thick] (0,0) circle (.12);
\draw (1.5,-.6) node {$b_{02}$};
\draw (.8,-1.2) node {$b_{02}$};
\draw (-.2,-1.2) node {$\dots$};
\draw (-1.3,-.6) node {$b_{02}$};
\draw (1.45,0) node {$b_{12}$};
\draw (-1.45,0) node {$b_{01}$};
\end{tikzpicture}
\end{align}
\begin{align}
\forall k\geq 0:\forall 1\leq \ell\leq k+1:\quad
&\BF^{[012]}_{k+1,0}(\un{b_{01}},\dots,\un{b_{01}},\underbrace{\un{c_{012}}}_{\ell\text{th}},\un{b_{02}},\dots,\un{b_{02}};e_0)(b_{01})=1,
\\
&\BF^{[012]}_{k+1,0}(\un{b_{01}},\dots,\un{b_{01}},\underbrace{\un{c_{012}}}_{\ell\text{th}},\un{b_{02}},\dots,\un{b_{02}};b_{01})(e_{1})=1,
\\
&\BF^{[012]}_{k+1,0}(\un{b_{02}},\dots,\un{b_{02}},\underbrace{\un{c_{012}}}_{\ell\text{th}},\un{b_{12}},\dots,\un{b_{12}};e_1)(b_{12})=1,
\\
&\BF^{[012]}_{k+1,0}(\un{b_{02}},\dots,\un{b_{02}},\underbrace{\un{c_{012}}}_{\ell\text{th}},\un{b_{12}},\dots,\un{b_{12}};b_{12})(e_{2})=1,
\\
\nonumber
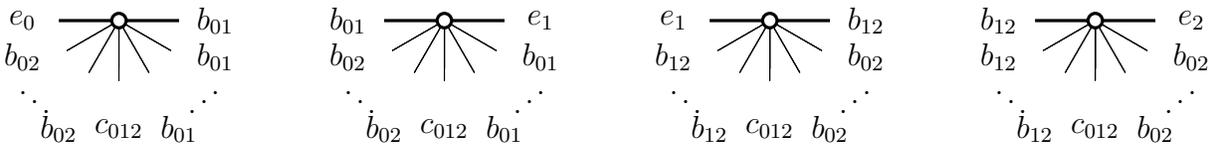
\begin{tikzpicture}[scale=.8, baseline=-1ex]
\draw
    (0,0) 
 -- (0,-.8)  -- (0,0)
 -- (-.866,-.5)  -- (0,0)
 -- (-.5,-.866)  -- (0,0)
 -- (0,-1)  -- (0,0)
 -- (.5,-.866)  -- (0,0) 
 -- (.866,-.5)  -- (0,0) ;
 \draw [very thick] (-1,0) -- (1,0);
\filldraw[color=black, fill=white, very thick] (0,0) circle (.12);
\draw (1.6,-.6) node {$b_{01}$};
\draw (1.4,-1.2) node {$\iddots$};
\draw (1,-1.8) node {$b_{01}$};
\draw (0,-1.8) node {$c_{012}$};
\draw (-1.6,-.6) node {$b_{02}$};
\draw (-1.4,-1.2) node {$\ddots$};
\draw (-1,-1.8) node {$b_{02}$};
\draw (1.6,0) node {$b_{01}$};
\draw (-1.6,0) node {$e_{0}$};
\end{tikzpicture}
\hspace{1cm}
\begin{tikzpicture}[scale=.8, baseline=-1ex]
\draw
    (0,0) 
 -- (0,-.8)  -- (0,0)
 -- (-.866,-.5)  -- (0,0)
 -- (-.5,-.866)  -- (0,0)
 -- (0,-1)  -- (0,0)
 -- (.5,-.866)  -- (0,0) 
 -- (.866,-.5)  -- (0,0) ;
 \draw [very thick] (-1,0) -- (1,0);
\filldraw[color=black, fill=white, very thick] (0,0) circle (.12);
\draw (1.6,-.6) node {$b_{01}$};
\draw (1.4,-1.2) node {$\iddots$};
\draw (1,-1.8) node {$b_{01}$};
\draw (0,-1.8) node {$c_{012}$};
\draw (-1.6,-.6) node {$b_{02}$};
\draw (-1.4,-1.2) node {$\ddots$};
\draw (-1,-1.8) node {$b_{02}$};
\draw (1.6,0) node {$e_1$};
\draw (-1.6,0) node {$b_{01}$};
\end{tikzpicture}
& \hspace{1cm}
\begin{tikzpicture}[scale=.8, baseline=-1ex]
\draw
    (0,0) 
 -- (0,-.8)  -- (0,0)
 -- (-.866,-.5)  -- (0,0)
 -- (-.5,-.866)  -- (0,0)
 -- (0,-1)  -- (0,0)
 -- (.5,-.866)  -- (0,0) 
 -- (.866,-.5)  -- (0,0) ;
 \draw [very thick] (-1,0) -- (1,0);
\filldraw[color=black, fill=white, very thick] (0,0) circle (.12);
\draw (1.6,-.6) node {$b_{02}$};
\draw (1.4,-1.2) node {$\iddots$};
\draw (1,-1.8) node {$b_{02}$};
\draw (0,-1.8) node {$c_{012}$};
\draw (-1.6,-.6) node {$b_{12}$};
\draw (-1.4,-1.2) node {$\ddots$};
\draw (-1,-1.8) node {$b_{12}$};
\draw (1.6,0) node {$b_{12}$};
\draw (-1.6,0) node {$e_1$};
\end{tikzpicture}
\hspace{1cm}
\begin{tikzpicture}[scale=.8, baseline=-1ex]
\draw
    (0,0) 
 -- (0,-.8)  -- (0,0)
 -- (-.866,-.5)  -- (0,0)
 -- (-.5,-.866)  -- (0,0)
 -- (0,-1)  -- (0,0)
 -- (.5,-.866)  -- (0,0) 
 -- (.866,-.5)  -- (0,0) ;
 \draw [very thick] (-1,0) -- (1,0);
\filldraw[color=black, fill=white, very thick] (0,0) circle (.12);
\draw (1.6,-.6) node {$b_{02}$};
\draw (1.4,-1.2) node {$\iddots$};
\draw (1,-1.8) node {$b_{02}$};
\draw (0,-1.8) node {$c_{012}$};
\draw (-1.6,-.6) node {$b_{12}$};
\draw (-1.4,-1.2) node {$\ddots$};
\draw (-1,-1.8) node {$b_{12}$};
\draw (1.6,0) node {$e_2$};
\draw (-1.6,0) node {$b_{12}$};
\end{tikzpicture}
\end{align}
\begin{multline}\label{EQU:F-c-c}
\forall k\geq 0:\forall 1\leq \ell_1<\ell_2\leq k+2:
\\
\BF^{[012]}_{k+2,0}(\un{b_{01}},\dots,\un{b_{01}},\underbrace{\un{c_{012}}}_{\ell_1\text{th}},\un{b_{02}},\dots,\un{b_{02}},\underbrace{\un{c_{012}}}_{\ell_2\text{th}},\un{b_{12}},\dots,\un{b_{12}};e_{1})(e_1)=1.
\end{multline}
\[
\begin{tikzpicture}[scale=.7, baseline=-1ex]
\draw
    (0,0) 
 -- (0.518,-1.932)  -- (0,0)
 -- (1.932,-.518)  -- (0,0)
 -- (1,-1.732)  -- (0,0)
 -- (1.732,-1)  -- (0,0)
 -- (1.414,-1.414)  -- (0,0)
 -- (0,-2)  -- (0,0)
  -- (-0.518,-1.932)  -- (0,0)
 -- (-1.932,-.518)  -- (0,0)
 -- (-1,-1.732)  -- (0,0)
 -- (-1.732,-1)  -- (0,0)
 -- (-1.414,-1.414)  -- (0,0);
 \draw [very thick] (-2,0) -- (2,0);
\filldraw[color=black, fill=white, very thick] (0,0) circle (.12);
\draw (2.7,-.7) node {$b_{01}$};
\draw (2.8,-1.3) node {$\vdots$};
\draw (2.4,-1.8) node {$b_{01}$};
\draw (2.3,-2.6) node {$c_{012}$};
\draw (.9,-2.8) node {$b_{02}$};
\draw (0,-2.8) node {$\dots$};
\draw (-.9,-2.8) node {$b_{02}$};
\draw (-2.3,-2.6) node {$c_{012}$};
\draw (-2.7,-.7) node {$b_{12}$};
\draw (-2.7,-1.3) node {$\vdots$};
\draw (-2.4,-1.8) node {$b_{12}$};
\draw (-2.7,0) node {$e_1$};
\draw (2.7,0) node {$e_{1}$};
\end{tikzpicture}
\]

We claim that the following equation holds, which will be proved in appendix \ref{APP:Proof-of-HIP-on-2-simplex}:
\begin{equation}\label{EQU:D(F012)=F01+F02+F12}
\delt\BF^{[012]}=\BF^{[01]}+\BF^{[02]}+\BF^{[12]}
\end{equation}

As before, the above will also be applied to obvious variations of the superscript, such as, e.g,  $A^{[123]}$ with maps $\BF^{[123]}$ given by replacing $0$, $1$ and $2$ in the above with $1$, $2$ and $3$, respectively, etc.
\end{ex}

\begin{ex}[$C^\bu(S^2;\Z_2)$ inner product with homotopies]\label{EX:local-CS2Z2}
We now use a tetrahedral triangulation of the $2$-sphere. More precisely, we set
\[
A=\bigoplus_{0\leq j\leq 3}\Z_2.e_j\oplus\bigoplus_{0\leq i<j\leq 3}\Z_2.b_{ij}\oplus\bigoplus_{0\leq i<j<l\leq 3}\Z_2.c_{ijl}
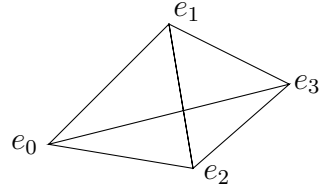
\hspace{1cm}
\begin{tikzpicture}[scale=.8, baseline=4ex]
\draw
    (0,0)  -- (2,2)  -- (2.4,-.4) -- (0,0)  -- (4,1) -- (2.4,-.4)  -- (2,2) -- (4,1);
\draw (-.4,0) node {$e_0$};
\draw (2.3,2.2) node {$e_1$};
\draw (2.8,-.5) node {$e_2$};
\draw (4.3,1) node {$e_3$};
\end{tikzpicture}
\]
with $|e_j|=0$, $|b_{ij}|=-1$ and $|c_{ijl}|=-2$ and differential $d(e_j)=\sum_{i\neq j} b_{\{ij\}}$, $d(b_{ij})=\sum_{l\neq i,j} c_{\{ijl\}}$, and $d(c_{ijl})=0$, where the bracket ``$\{\}$'' denotes indices in ascending order, i.e., $b_{\{ij\}}=\left\{\begin{matrix}b_{ij}\text{, for } i<j\\b_{ji}\text{, for } j<i\end{matrix}\right.$ and similarly for $c_{\{ijl\}}$. The multiplication is similar to \eqref{EQU:Product-of-2-simplex}, where the indices ``$012$'' are replaced by any $0\leq i<j<l\leq 3$:
\begin{align}\label{EQU:Product-of-CS2Z2}
&\forall j:\quad
e_j\cdot e_j=e_j, \quad
\forall i<j:\quad
e_i\cdot b_{ij}=b_{ij}, \quad
b_{ij}\cdot e_j=b_{ij}, \\ \nonumber
&\forall i<j<l:\quad
 e_i\cdot c_{ijl}=c_{ijl}, \quad
c_{ijl}\cdot e_l=c_{ijl}, \quad
b_{ij}\cdot b_{jl}=c_{ijl},
\end{align}
and all other multiplications vanish. Note in particular that $A$ has the unit $e_0+e_1+e_2+e_3$.

We define the homotopy inner product $\BF$ for $\B{A}:=A$ to be given by \eqref{EQU:Fk0-02-2-012}-\eqref{EQU:F-c-c} with ``$012$'' in \eqref{EQU:Fk0-02-2-012}-\eqref{EQU:F-c-c} replaced by any $0\leq i<j<l\leq 3$, i.e.,
\begin{equation}\label{EQU:F-on-CS2Z2}
\BF:=\BF^{[012]}+\BF^{[013]}+\BF^{[023]}+\BF^{[123]}.
\end{equation}

{\bf Claim:} $\BF$ is a homotopy inner product, i.e., $\delt \BF=0$.
\begin{proof}
We compute $\delt \BF=\delt \BF^{[012]}+\delt\BF^{[013]}+\delt\BF^{[023]}+\delt\BF^{[123]}$. We claim that we can use equation \eqref{EQU:D(F012)=F01+F02+F12} to evaluate this expression, which is not completely obvious since $\delt$ in  \eqref{EQU:D(F012)=F01+F02+F12} is for the dg-algebra $A^{[012]}$, whereas we need to apply $\delt$ for $A$.

To see that \eqref{EQU:D(F012)=F01+F02+F12} holds for the dg-algebra $A$, note that, for example, $\delt\BF^{[012]}(\un{a_1},\dots)(a_r)$ applies a differential or a multiplication to the inputs $a_j\in A$ (as stated in  \eqref{EQU:de0(F)} and \eqref{EQU:de1(F)}). Now, if any of the inputs $a_j$ are generators of $A$ which has a $3$ in the index (such as, e.g., $b_{13}\in A$), then taking the differential or a multiplication will consist of sums of generators with $3$ in their indices, and thus $\delt\BF^{[012]}$ applied to it would vanish. On the other hand, if all inputs $a_j$ applied to $\delt\BF^{[012]}$ are only given by generators indexed by $0$, $1$, and $2$, then $\delt\BF^{[012]}$ applied to it coincides with what we would get for $A^{[012]}$, because, the multiplication in $A$ equals the one in $A^{[012]}$, and the differential in $A$ equals the differential in $A^{[012]}$ up to generators with indices of $3$. (For example, applying $d(b_{12})=c_{012} +c_{123}$ in $\delt\BF^{[012]}$ will vanish on $c_{123}$ and will coincide with \eqref{EQU:D(F012)=F01+F02+F12} on $c_{012}=d_{A^{[012]}}(b_{12})=$(differential of $b_{12}$ in $A^{[012]}$).)

Thus, $\delt \BF=\delt \BF^{[012]}+\delt\BF^{[013]}+\delt\BF^{[023]}+\delt\BF^{[123]}=(\BF^{[01]}+\BF^{[02]}+\BF^{[12]})+(\BF^{[01]}+\BF^{[03]}+\BF^{[13]})+(\BF^{[02]}+\BF^{[03]}+\BF^{[23]})+(\BF^{[12]}+\BF^{[13]}+\BF^{[23]})=0$.
\end{proof}
\end{ex}

\begin{prop}\label{PROP:pullback-of-local}
Let $A$ be the dg-algebra from example \ref{EX:local-CS2Z2}, and let $B=H^\bu(S^2;\Z_2)\cong \Z_2.e\oplus \Z_2.s$ be the dg-algebra from examples \ref{EX:HS2Z2-no-homotopies} and \ref{EX:HS2Z2-with-homotopies}. Then, the map $f:B\to A$, $f(e):=e_0+e_1+e_2+e_3$ and $f(s):=c_{012}$, is a dg-algebra quasi-isomorphism.

Transferring the homotopy inner product $\BF\ov{A}$ defined in \eqref{EQU:F-on-CS2Z2} via the map $f$ gives $f(\BF)\ov{B}=\wt \BF\ov{B}$, where $\wt \BF\ov{B}$ is the homotopy inner product defined in equation \eqref{EQU:HIP-in-H(S2,Z2)} from example \ref{EX:HS2Z2-with-homotopies}.
\end{prop}
\begin{proof}
The transfered map of $\BF\ov{A}$ from \eqref{EQU:F-on-CS2Z2} is given by applying $f$ to inputs from $B$ and then applying $\BF\ov{A}$; see \eqref{EQU:f(F)/B-as-F/A}. Since the image of $f$ is spanned by $c_{012}$ and $e_0+e_1+e_2+e_3$, it follows from \eqref{EQU:Fk0-02-2-012}-\eqref{EQU:F-c-c} that the only non-zero components (after applying $f$) are:
\[
\BF^{[012]}_{0,0}(c_{012})(e_2)=1, \quad 
\BF^{[012]}_{0,0}(e_0)(c_{012})=1, \quad
\BF^{[012]}_{2,0}(\un{c_{012}},\un{c_{012}};e_1)(e_1)=1.
\]
Thus, we get precisely the non-zero components of $\wt \BF$ from \eqref{EQU:HIP-in-H(S2,Z2)} for $f(\BF)\ov{B}$.
\end{proof}

\section{Computations of the BV algebras}\label{SEC:compute-Hochschild-BV}

We now compute the BV algebras on Hochschild cohomology induced by the homotopy inner products considered in examples \ref{EX:HS2Z2-no-homotopies} and \ref{EX:HS2Z2-with-homotopies} in section \ref{SEC:compute-bimodule-maps}. In particular, we complete the computations from theorem \ref{THM:main-theorem} for the $2$-sphere by computing $HH^\bu(H^\bu,\B{H}^\bu)$ and $HH^\bu(H^\bu,\B{H}_\bu)$ for $H^\bu=\B{H}^\bu=H^\bu(S^2;\Z_2)$ and $\B{H}_\bu=H_\bu(S^2;\Z_2)$, and we will check the formulas for the maps $\CF$ and $\wt\CF$ in \eqref{EQU:Hoch-bimodule-map-H} stated in theorem \ref{THM:main-theorem}. In the section the ground ring is $R=\Z_2$.

\subsection{Hochschild cohomology $HH^\bu(H^\bu,\B{H}^\bu)$ and the cup product}\label{SUBSEC:HH(H,H)-and-cup}

Denote by $H^\bu=H^\bu(S^2;\Z_2)\cong  \Z_2.e\oplus \Z_2.s$ from example \ref{EX:HS2Z2-no-homotopies} with $|e|=0$ and $|s|=-2$. The generators of normalized Hochschild cochains $\NCH^\bu(H^\bu,\B{H}^\bu)$ are $\phi_k$, $\psi_k$ for $k\geq 0$ given by
\[
\phi_k(\underbrace{\un{s},\dots,\un{s}}_{k\text{ many}}):=e, \quad\quad \psi_k(\underbrace{\un{s},\dots,\un{s}}_{k\text{ many}}):=s.
\]
Note that $|\phi_k|=k$, $|\psi_k|=k-2$, since the shifted generator $s$ is of degree $|\un{s}|=-1$.
Applying the Hochschild differential (see \ref{SUBSEC:basic-setup}, and using $d=0$ and $s\cdot s=0$, $s\cdot e=e\cdot s=s$), we see that all $\phi_k$ and $\psi_k$ are closed (over $\Z_2$). Thus,
\[
HH^\bu(H^\bu,\B{H}^\bu)\cong\bigoplus_{k\geq 0} \Z_2.\phi_k\oplus \bigoplus_{k\geq 0} \Z_2.\psi_k.
\]

The cup product \eqref{EQU:cup-product} is readily checked to be $\phi_k\smile \phi_\ell=\phi_{k+\ell}$, $\psi_k\smile \psi_\ell=0$, and $\phi_k\smile \psi_\ell=\psi_\ell\smile \phi_k=\psi_{k+\ell}$.

\subsection{Hochschild cohomology $HH^\bu(H^\bu,\B{H}_\bu)$ and Connes' $\CB$-operator}\label{SUBSEC:HH(H,H*)-and-B}

Denote by $\B{H}_\bu=H_\bu(S^2;\Z_2)\cong\Z_2.e^*\oplus \Z_2.s^*$ from example \ref{EX:HS2Z2-no-homotopies}, with $|e^*|=0$ and $|s^*|=2$ and  $s.e^*=e^*.s=0, s.s^*=s^*.s=e^*$. The generators of normalized Hochschild cochains $\NCH^\bu(H^\bu,\B{H}_\bu)$ are $\theta_k$, $\chi_k$ for $k\geq 0$ given by
\[
\theta_k(\underbrace{\un{s},\dots,\un{s}}_{k\text{ many}}):=s^*, \quad\quad \chi_k(\underbrace{\un{s},\dots,\un{s}}_{k\text{ many}}):=e^*.
\]
Note that $|\theta_k|=k+2$, $|\chi_k|=k$. All $\theta_k$ and $\chi_k$ are closed under the Hochschild cochain differential (over $\Z_2$), so that
\[
HH^\bu(H^\bu,\B{H}_\bu)\cong\bigoplus_{k\geq 0} \Z_2.\theta_k\oplus \bigoplus_{k\geq 0} \Z_2.\chi_k.
\]

Whereas $HH^\bu(H^\bu,\B{H}^\bu)$ has the cup product \eqref{EQU:cup-product}, $HH^\bu(H^\bu,\B{H}_\bu)$ has Connes' $\CB$-operator \eqref{EQU:B-operator-Def}. We compute $\CB$ from \eqref{EQU:B-operator-Def} to be $(\CB(\varphi))(\underbrace{\un{s},\dots, \un{s}}_{(r-1) \text{ many}})(a):=r\cdot \varphi(\un{s},\dots,\un{s}, \un{a},\un{s},\dots, \un{s})(e)$, which applied to $\theta_k$ and $\chi_k$ yields:
\[
\CB(\theta_r)=0, \quad \CB(\chi_r)=r\cdot \theta_{r-1}.
\]
This confirms the expressions for $\CB$ stated in theorem \ref{THM:main-theorem}.

\subsection{$\CF$ induced by example \ref{EX:HS2Z2-no-homotopies} and the BV delta $\Delta^\BF$}\label{SUBSEC:Delta-F-strict}

The dg-bimodule map $\BF:\B{H}^\bu\to \B{H}_\bu$, $\BF(s)=e^*$ and $\BF(e)=s^*$, from example \ref{EX:HS2Z2-no-homotopies} induced a map on Hochschild cohomologies $\CF:HH^\bu(H^\bu,\B{H}^\bu)\to HH^\bu(H^\bu,\B{H}_\bu)$ via \eqref{EQU:CH(F)}. Since $\BF$ only has a lowest component $\BF_{0,0}$ this is simply
\[
\CF(\phi_k)=\BF\circ \phi_k=\theta_k, \quad \CF(\psi_k)=\BF\circ \psi_k=\chi_k.
\]
$\CF$ is clearly an isomorphism, and confirms $\CF$ from theorem \ref{THM:main-theorem}.

We also compute $\Delta^\BF$ and check it is equal to $\CB$ transferred to $HH^\bu(H^\bu,\B{H}^\bu)$, i.e., condition \eqref{ITEM:B=Delta} from definition \ref{DEF:PD-structure} for a Poincar\'e duality structure. From \eqref{EQU:DEF-of-ZF}, we get for $\varphi\in Hom(\un{H^\bu}^{\otimes r}, \B{H}^\bu)$ that $\CZ^\BF(\varphi)(\un{s},\dots,\un{s})(a)=r\cdot \BF(e)(\varphi(\un{s},\dots,\un{a},\dots,\un{s}))$, which gives
\[
 \CZ^\BF(\phi_k)=0=\CB(\theta_k)=\CB\circ\CF(\phi_k), \quad
 \CZ^\BF(\psi_k)=k\cdot \theta_{k-1}=\CB(\chi_k)=\CB\circ\CF(\psi_k).
\]
Therefore, $\CF^{-1}\circ \CB \circ \CF=\CF^{-1}\circ \CZ^\BF=\De^\BF$.

\subsection{$\wt\CF$ induced by example \ref{EX:HS2Z2-with-homotopies} and the BV delta $\Delta^{\wt\BF}$}\label{SUBSEC:Delta-F-homotopies}

Consider the dg-bimodule map $\wt\BF:\B{H}^\bu\to \B{H}_\bu$ from example \ref{EX:HS2Z2-with-homotopies}. Since $\wt\BF=\wt \BF_{0,0}+\wt \BF_{2,0}$ is $\wt\BF_{0,0}(s)=e^*$, $\wt\BF_{0,0}(e)=s^*$ with one extra homotopy $\wt\BF_{2,0}$ which is non-zero only for $\wt\BF_{2,0}(\un{s},\un{s};e):=e^*$ (see \eqref{EQU:HIP-in-H(S2,Z2)}), we get
\[
\wt\CF(\phi_k)=\wt\BF_{0,0}\circ \phi_k+\wt\BF_{2,0}(-,-;\phi_k(-))=\theta_k+\chi_{k+2}, \quad \wt\CF(\psi_k)=\wt\BF_{0,0}\circ \psi_k=\chi_k.
\]
This confirms $\wt\CF$ from theorem \ref{THM:main-theorem}. As was noted in theorem \ref{THM:main-theorem}, $\wt\CF$ is an isomorphism with inverse $\wt{\CF}^{-1}$ given by $\wt{\CF}^{-1}(\theta_k)=\phi_k+\psi_{k+2}$ and $\wt{\CF}^{-1}(\chi_k)=\psi_k$. 

We also check that in this case $\wt\CF^{-1}\circ \CB \circ\wt\CF=\De^{\wt\BF}$. From \eqref{EQU:DEF-of-ZF}, we get for $\varphi\in Hom(\un{H^\bu}^{\otimes r}, \B{H}^\bu)$, that
\begin{align*}
\CZ^{\wt \BF}(\varphi)(\un{s},\dots,\un{s})(a)=&
r\cdot \wt\BF_{0,0}(e)(\varphi(\un{s},\dots,\un{a},\dots,\un{s}))
+r\cdot  \wt\BF_{2,0}(\un{s},\un{s};e)(\varphi(\un{s},\dots,\un{a},\dots,\un{s}))
\\
&+\wt\BF_{2,0}(\un{\varphi(\un{s},\dots,\un{s})},\un{s};e)(a)
+\wt\BF_{2,0}(\un{s},\un{\varphi(\un{s},\dots,\un{s})};e)(a)
\end{align*}
which gives
\[
 \CZ^{\wt \BF}(\phi_k)=0+k\cdot \theta_{k+1}+0+0=k\cdot \theta_{k+1}
, \quad
 \CZ^{\wt \BF}(\psi_k)=k\cdot \theta_{k-1}+0+\chi_{k+1}+\chi_{k+1}=k\cdot \theta_{k-1}.
 \quad
\]
Since $\CB \circ\wt\CF(\phi_{k})=\CB(\theta_k+\chi_{k+2})=k\cdot \theta_{k+1}$ and $\CB \circ\wt\CF(\psi_{k})=\CB(\chi_k)=k\cdot \theta_{k-1}$, this shows that $ \CZ^{\wt \BF}=\CB \circ\wt\CF$ and, thus, ${\wt \CF}^{-1}\circ \CB \circ {\wt \CF}={\wt \CF}^{-1}\circ \CZ^{\wt \BF}=\De^{\wt \BF}$. Explicitly, $\De^{\wt \BF}$ is given by:
\begin{align*}
\De^{\wt \BF}(\phi_k)& ={\wt \CF}^{-1}\circ \CZ^{\wt \BF}(\phi_k)={\wt \CF}^{-1}(k\cdot \theta_{k+1})=k\cdot(\phi_{k+1}+\psi_{k+3})
\\
\De^{\wt \BF}(\psi_k)& ={\wt \CF}^{-1}\circ \CZ^{\wt \BF}(\psi_k)={\wt \CF}^{-1}(k\cdot \theta_{k-1})=k\cdot(\phi_{k-1}+\psi_{k+1})
\end{align*}
This confirms $\wt{\De}$ from \eqref{EQU:tilde-Delta-on-HH} and thus theorem \ref{THM:main-theorem}.

\subsection{Example of a homotopy inner product that is not a Poincar\'e duality structure}\label{SUBSEC:PD-counterexample}
We end this section with an example of a homotopy inner product which is not a Poincar\'e duality structure due to its failure of condition \eqref{ITEM:B=Delta} from definition \ref{DEF:PD-structure} (while condition \eqref{ITEM:Fsharp-HH-iso} holds). We therefore do not get a BV algebra on Hochschild cohomology $HH^\bu(A,\B{A})$ with underlying algebra $A$, since there is no $\Delta$ operator, but only an operator, $\CF^{-1}\circ \CB \circ \CF$, that squares to zero, and a \emph{different} operator, $\De^{\BF}$, whose deviation from being a derivation is a Gerstenhaber bracket; cf.  lemma \ref{LEM:BV-Delta-conditions}.

Let $A=\Z_2.e\oplus \Z_2.b\oplus \Z_2.c$ with zero differential, $|e|=0$, $|b|=-1$, $|c|=-2$, and where $e$ is the unit, $b\cdot b=c$ and all other products vanish. We denote $\B{A}=A$ and its dual $\B{A}^*=\Z_2.e^*\oplus \Z_2.b^*\oplus \Z_2.c^*$ where $e^*,  b^*, c^*$ are the duals of $e,b,c$, respectively, so that $|e^*|=0$, $|b^*|=1$, $|c^*|=2$. Define the homotopy inner product $\BF$ whose only non-vanishing components are $\BF_{0,0}$ and $\BF_{3,0}$ given by
\begin{align*}
& \BF_{0,0}(c)(e)= \BF_{0,0}(b)(b)= \BF_{0,0}(e)(c)=1,\\
& \BF_{3,0}(\un{c},\un{b},\un{c};e)(e)= \BF_{3,0}(\un{b},\un{b},\un{b};c)(e)= \BF_{3,0}(\un{b},\un{b},\un{b};e)(c)= \BF_{3,0}(\un{b},\un{b},\un{b};b)(b)=1.
\end{align*}
One can check explicitly that the conditions for a homotopy inner product $0=\delt \BF=\delt_1 \BF$, where $\delt_1$ is given by \eqref{EQU:de1(F)}, are satisfied.

Note, that $\BF$ induces an isomorphism on Hochschild cochains $\CF:\NCH^\bu(A,\B{A})\stackrel{\cong}\longrightarrow \NCH^\bu(A,\B{A}^*)$, since $\BF_{0,0}:\B{A}\to \B{A}^*$ is an isomorphism, and thus $\BF$ satisfies \ref{DEF:PD-structure} condition \eqref{ITEM:Fsharp-HH-iso}.

We next show that \ref{DEF:PD-structure} condition \eqref{ITEM:B=Delta} fails, i.e., that $ \CB \circ \CF\neq \CZ^\BF$. Consider $\varphi\in Hom(\un{A}^{\otimes 0},\B{A})$, $\varphi(1)=e$. Note that $\varphi$ is a closed element in the Hochschild cochains $\varphi\in \NCH^\bu(A,\B{A})$. Plugging $\varphi$ into \eqref{EQU:DEF-of-ZF} gives $\CZ^\BF(\varphi)=0$, since each individual summand of \eqref{EQU:DEF-of-ZF} vanishes. Next, to compute $\CF(\varphi)$, we see the only non-zero evaluation of $\CF(\varphi)$ are (the ones with ``$e$'' in the input spot for the module of $\BF$):
\begin{align*}
&\CF(\varphi)(1)=\BF_{0,0}(\varphi(1))=c^*\\
&\CF(\varphi)(\un{c},\un{b},\un{c})=\BF_{3,0}(\un{c},\un{b},\un{c};\varphi(1))=e^*\\
&\CF(\varphi)(\un{b},\un{b},\un{b})=\BF_{3,0}(\un{b},\un{b},\un{b};\varphi(1))=c^*
\end{align*}
Applying $\CB$ from \eqref{EQU:B-operator-Def} (note that the unit in $A$ is in this example written as $e$) to this is only non-zero on the middle term, for which we get the only non-zero evaluations:
\[
\CB(\CF(\varphi))(\un{c},\un{b})(c)=1, \quad
\CB(\CF(\varphi))(\un{b},\un{c})(c)=1, \quad
\CB(\CF(\varphi))(\un{c},\un{c})(b)=1.
\]
Then, $\CB(\CF(\varphi))$ is non-vanishing, and closed in $\NCH^\bu(A,\B{A}^*)$, which follows either because $\varphi$ is closed and $\CB\circ\CF$ is a chain map or by direct inspection.

{\bf Claim:} $\CB(\CF(\varphi))$ is not exact in $\NCH^\bu(A,\B{A}^*)$.

The claim implies that $\CB \circ \CF(\varphi)\neq 0=\CZ^\BF(\varphi)$ in $HH^\bu(A,\B{A}^*)$, and so $\BF$ is not a Poincar\'e duality structure as it fails condition \eqref{ITEM:B=Delta} from definition \ref{DEF:PD-structure}.

\begin{proof}[Proof of the claim:]
To prove the claim, note that $d=0$ in $A$ implies that the Hochschild differential $\delt=\delt_0+\delt_1=\delt_1$ from \ref{SUBSEC:basic-setup} must increase the number of tensor factors of $\un{A}$ for the inputs. Thus any $\wt \varphi$ which makes $\CB(\CF(\varphi))$ exact, i.e., $\delt(\wt \varphi)=\CB(\CF(\varphi))$, must have non-vanishing components with exaclty one tensor input. Since the degree $|\CB(\CF(\varphi))|=3$, it must be that $|\wt \varphi|=4$. However, all Hochschild cochains in $\NCH^\bu(A,\B{A}^*)$ with one or no input are of degree $\leq 3$. Thus, $\CB(\CF(\varphi))$ cannot be exact.
\end{proof}

\appendix

\section{Proof of equations \eqref{EQU:D(F01)=F0+F1} and \eqref{EQU:D(F012)=F01+F02+F12}}\label{APP:Proof-of-HIP-on-2-simplex}

In this appendix we prove the identity \eqref{EQU:D(F01)=F0+F1}, i.e., $\delt \BF^{[01]}=\BF^{[0]}+\BF^{[1]}$, for the maps $\BF^{[01]}$ given by \eqref{EQU:Fk0-for-1-simplicies} from example \ref{EX:F-for-the-1-simplex}, as well as the identity \eqref{EQU:D(F012)=F01+F02+F12}, i.e., $\delt\BF^{[012]}=\BF^{[01]}+\BF^{[02]}+\BF^{[12]}$, for the maps $\BF^{[012]}$ given by \eqref{EQU:Fk0-02-2-012}-\eqref{EQU:F-c-c} from example \ref{EX:F-for-the-2-simplex}. The ground ring in this appendix is $R=\Z_2$.

\subsection{Preliminaries}

Any $\BF_{p,q}$ is a map $\BF_{p,q}:\un{A}^{\ot p}\ot \B{A}\ot \un{A}^{\ot q}\to \B{A}^*$, which is an element of $\BF_{p,q}\in (\un{A}^*)^{\ot p}\ot \B{A}^*\ot (\un{A}^*)^{\ot q}\ot \B{A}^*$. Consider $A^{[01]}$ with generators $e_0, e_1, b_{01}$, respectively $A^{[012]}$ with generators $e_0, e_1, e_2, b_{01}, b_{12}, b_{02}, c_{012}$). For a multi-index $I\in \{0, 1, 2, 01, 02, 12, 012\}$, define $a^*_I$ to be $a^*_0:=e^*_0$, $a^*_1:=e^*_1$, $a^*_2:=e^*_2$, $a^*_{01}:=b^*_{01}$, $a^*_{02}:=b^*_{02}$, $a^*_{12}:=b^*_{12}$, $a^*_{012}:=c^*_{012}$. With this we use the following notation, for a given sequence of multi-indices $I_1,\dots, I_{p+q+2}$:
\begin{multline}
\hip{I_1,\dots,I_p}{I_{p+1}}{I_{p+2},\dots, I_{p+q+1}}{I_{p+q+2}}
\\
:= a^*_{I_1}\ot\dots \ot a^*_{I_{p}}\ot a^*_{I_{p+1}}\ot a^*_{I_{p+2}}\ot\dots\ot a^*_{I_{p+q+1}}\ot a^*_{I_{p+q+2}}
 \in (\un{A}^*)^{\ot p}\ot \B{A}^*\ot (\un{A}^*)^{\ot q}\ot \B{A}^*
\end{multline}
Moreover, placing a tilde under the multi-index sums over any number of these indices, i.e., for tildes at positions $1\leq k_1<k_2<\dots<k_r\leq p$, we define
\begin{multline}
\hip{I_1,\dots,\rep{I_{k_1}}, \dots,\rep{I_{k_2}}, \dots,I_{p}}{I_{p+1}}{I_{p+2},\dots, I_{p+q+1}}{I_{p+q+2}}
\\
:=\sum_{n_1,n_2,\dots,n_r \geq 0}
\hip{I_1,\dots,\underbrace{{I_{k_1}},\dots,{I_{k_1}}}_{n_1\text{ many}}, \dots,\underbrace{{I_{k_2}},\dots,{I_{k_2}}}_{n_2\text{ many}}, \dots,I_{p}}{I_{p+1}}{I_{p+2},\dots, I_{p+q+1}}{I_{p+q+2}}
\end{multline}
Since our examples are chosen to be ``strict'' on the right (see observation \ref{OBS:not-bimodule}), we only need higher homotopies (and thus tildes) on the first $p$ tensor factors, but not on the other $q$ tensor factors of $\un{A}^*$.

We apply $\delt=\delt_0+\delt_1$ from \eqref{EQU:de0(F)} and \eqref{EQU:de1(F)} to the above. Here, $\delt_0$ applies the differential of $A^*$, which, in the multi-indiex notation, removes one index, (for example, $01\mapsto 0+1$, $012\mapsto 01+02+12$, etc.). $\delt_1$ applies a multiplication, which on $A^*$ applies the Alexander-Whitney coproduct (for example, $01\mapsto 0\ot 01+01\ot 1$, $012\mapsto 0\ot 012+01\ot12+012\ot 2$, etc.); see also the calculation below.

\subsection{Proof of \eqref{EQU:D(F01)=F0+F1}}
$\BF^{[01]}$ from \eqref{EQU:Fk0-for-1-simplicies} is given by $\BF^{[01]}=\hip{\rep{01}}{0}{}{01}+\hip{\rep{01}}{01}{}{1}$. We calculate $\delt=\delt_0+\delt_1$ applied to $\BF^{[01]}$.
\begin{align*}
&\delt_0 (\hip{\rep{01}}{0}{}{01})
=&\hip{\rep{01}, 0, \rep{01}}{0}{}{01}\cancl{01}  &&+ \hip{\rep{01},1, \rep{01}}{0}{}{01}\cancl{02}
\\ &&\,\,+ \hip{\rep{01}}{0}{}{0}\cancl{03}  &&+ \hip{\rep{01}}{0}{}{1}\cancl{04}
\\
&\delt_1 (\hip{\rep{01}}{0}{}{01})
=& \hip{\rep{01}, 0,01,\rep{01}}{0}{}{01}\cancl{01}  &&+ \hip{\rep{01},01,1,\rep{01}}{0}{}{01}\cancl{02}
\\&&+ \hip{\rep{01},0}{0}{}{01}\cancl{01}  &&+ \hip{\rep{01}}{0}{0}{01}\cancl{05}
\\&&+ \hip{\rep{01}}{0}{0}{01}\cancl{05}  &&+ \hip{\rep{01}}{0}{01}{1}\cancl{06}
\\&&+ \hip{01,\rep{01}}{0}{}{0}\cancl{03}  &&+ \hip{1,\rep{01}}{0}{}{01}\cancl{02}
\\
&\delt_0 (\hip{\rep{01}}{01}{}{1})
=& \hip{\rep{01},0,\rep{01}}{01}{}{1}\cancl{07}  &&+ \hip{\rep{01},1,\rep{01}}{01}{}{1}\cancl{08}
\\ &&\,\,+ \hip{\rep{01}}{0}{}{1}\cancl{04}  &&+ \hip{\rep{01}}{1}{}{1}\cancl{09}
\\
&\delt_1 (\hip{\rep{01}}{01}{}{1})
=& \hip{\rep{01},0,01,\rep{01}}{01}{}{1}\cancl{07}  &&+ \hip{\rep{01},01,1,\rep{01}}{01}{}{1}\cancl{08}
\\&&+ \hip{\rep{01},0}{01}{}{1}\cancl{07}  &&+ \hip{\rep{01},01}{1}{}{1}\cancl{09}
\\&&+ \hip{\rep{01}}{0}{01}{1}\cancl{06}  &&+ \hip{\rep{01}}{01}{1}{1}\cancl{10}
\\&&+ \hip{\rep{01}}{01}{1}{1}\cancl{10} && + \hip{1,\rep{01}}{01}{}{1}\cancl{08}
\end{align*}
The gray box on the right of each expression indicates which terms can be combined. Note that for $\cancl{01}$, $\cancl{02}$, $\cancl{07}$, and $\cancl{08}$, there are always 3 terms that can be combined and cancel. In fact, all terms cancel, except for terms labeled $\cancl{03}$, which has a left over term of $\hip{}{0}{}{0}$, and terms labeled $\cancl{09}$, which has a left over term of $\hip{}{1}{}{1}$. These two left-over terms are $\BF^{[0]}$ and $\BF^{[1]}$, which shows that $\delt \BF^{[01]}=\BF^{[0]}+\BF^{[1]}$, i.e., equation \eqref{EQU:D(F01)=F0+F1}.

\subsection{Proof of \eqref{EQU:D(F012)=F01+F02+F12}}

$\BF^{[012]}$ from \eqref{EQU:Fk0-02-2-012}-\eqref{EQU:F-c-c} is given by
\begin{align*}
\BF^{[012]}=&\hip{\rep{02}}{012}{}{2}+\hip{\rep{02}}{0}{}{012}+\hip{\rep{02}}{01}{}{12}
\\
&+\hip{\rep{01},012,\rep{02}}{0}{}{01}+\hip{\rep{01},012,\rep{02}}{01}{}{1}
\\
&+\hip{\rep{02},012,\rep{12}}{1}{}{12}+\hip{\rep{02},012,\rep{12}}{12}{}{2}
\\
&+\hip{\rep{01},012,\rep{02},012,\rep{12}}{1}{}{1}
\end{align*}
We calculate $\delt=\delt_0+\delt_1$ applied to $\BF^{[012]}$. Again, the gray box indicates which terms combine and cancel. Left-over terms will be collected below.
\begin{align*} %--------------------------------3 cases-------------------------
\delt_0 (\hip{\rep{02}}{012}{}{2})=
&& \hip{\rep{02},0,\rep{02}}{012}{}{2}\cancl{01}    &&+ \hip{\rep{02},2,\rep{02}}{012}{}{2}\cancl{02}
\\&&+\hip{\rep{02}}{01}{}{2}\cancl{03}     &&+  \hip{\rep{02}}{02}{}{2}\cancl{04}
\\&&  +\hip{\rep{02}}{12}{}{2}\cancl{05}
\\
\delt_1 (\hip{\rep{02}}{012}{}{2})=
&& \hip{\rep{02},0,02,\rep{02}}{012}{}{2}\cancl{01}  && +\hip{\rep{02},02,2,\rep{02}}{012}{}{2}\cancl{02} 
\\&& +\hip{\rep{02},0}{012}{}{2}\cancl{01}  && +\hip{\rep{02},01}{12}{}{2}\cancl{06} 
\\&& +\hip{\rep{02},012}{2}{}{2}\cancl{07}  && +\hip{\rep{02}}{012}{2}{2}\cancl{08} 
\\&& +\hip{\rep{02}}{01}{12}{2}\cancl{09}  && +\hip{\rep{02}}{0}{012}{2}\cancl{10} 
\\&& +\hip{\rep{02}}{012}{2}{2}\cancl{08}  && +\hip{2,\rep{02}}{012}{}{2}\cancl{02} 
\\
\delt_0 (\hip{\rep{02}}{0}{}{012})=
&&  \hip{\rep{02},0,\rep{02}}{0}{}{012}\cancl{11}    && +\hip{\rep{02}, 2,\rep{02}}{0}{}{012}\cancl{12}    
\\&&+ \hip{\rep{02}}{0}{}{01}\cancl{13}    && + \hip{\rep{02}}{0}{}{02}\cancl{14}
\\&&+\hip{\rep{02}}{0}{}{12}\cancl{15}    
\\
\delt_1 (\hip{\rep{02}}{0}{}{012})=
&& \hip{\rep{02},0,02,\rep{02}}{0}{}{012}\cancl{11} && +\hip{\rep{02},02,2,\rep{02}}{0}{}{012}\cancl{12}
\\&& \hip{\rep{02},0}{0}{}{012}\cancl{11} && +\hip{\rep{02}}{0}{0}{012}\cancl{16}
\\&& \hip{\rep{02}}{0}{0}{012}\cancl{16} && +\hip{\rep{02}}{0}{01}{12}\cancl{17}
\\&& \hip{\rep{02}}{0}{012}{2}\cancl{10} && +\hip{2,\rep{02}}{0}{}{012}\cancl{12}
\\&& \hip{12,\rep{02}}{0}{}{01}\cancl{18} && +\hip{012,\rep{02}}{0}{}{0}\cancl{19}
\\
\delt_0 (\hip{\rep{02}}{01}{}{12})=
&&  \hip{\rep{02},0,\rep{02}}{01}{}{12}\cancl{20}   && +\hip{\rep{02},2,\rep{02}}{01}{}{12}\cancl{21}   
\\&&+ \hip{\rep{02}}{0}{}{12}\cancl{15}   && +\hip{\rep{02}}{1}{}{12}\cancl{22}   
\\&&+\hip{\rep{02}}{01}{}{1}\cancl{23}   && + \hip{\rep{02}}{01}{}{2}\cancl{03}   
\\
\delt_1 (\hip{\rep{02}}{01}{}{12})=
&&\hip{\rep{02},0,02,\rep{02}}{01}{}{12}\cancl{20}  && + \hip{\rep{02},02,2,\rep{02}}{01}{}{12}\cancl{21}
\\&& +\hip{\rep{02},0}{01}{}{12}\cancl{20}  && + \hip{\rep{02},01}{1}{}{12}\cancl{24}
\\&& +\hip{\rep{02}}{01}{1}{12}\cancl{25}  && + \hip{\rep{02}}{0}{01}{12}\cancl{17}
\\&& +\hip{\rep{02}}{01}{1}{12}\cancl{25}  && + \hip{\rep{02}}{01}{12}{2}\cancl{09}
\\&& +\hip{2,\rep{02}}{01}{}{12}\cancl{21}  && + \hip{12,\rep{02}}{01}{}{1}\cancl{26}
\end{align*}
\begin{align*} %--------------------------------4 cases-------------------------
 \delt_0 (\hip{\rep{01},012,\rep{02}}{0}{}{01})=\hspace{-5cm}
\\ &&  \hip{\rep{01},0,\rep{01},012,\rep{02}}{0}{}{01}\cancl{27}   && +\hip{\rep{01},1,\rep{01},012,\rep{02}}{0}{}{01}\cancl{28}
\\&&+ \hip{\rep{01},01,\rep{02}}{0}{}{01}\cancl{13}&& +\hip{\rep{01},02,\rep{02}}{0}{}{01}\cancl{13}
\\&&+ \hip{\rep{01},12,\rep{02}}{0}{}{01}\cancl{18}
\\ && + \hip{\rep{01},012,\rep{02},0,\rep{02}}{0}{}{01}\cancl{29} &&+\hip{\rep{01},012,\rep{02},2,\rep{02}}{0}{}{01}\cancl{30}
\\ &&+ \hip{\rep{01},012,\rep{02}}{0}{}{0}\cancl{19} && +\hip{\rep{01},012,\rep{02}}{0}{}{1}\cancl{31}
\\
\delt_1 (\hip{\rep{01},012,\rep{02}}{0}{}{01})=\hspace{-5cm}
\\&& \hip{\rep{01},0,01,\rep{01},012,\rep{02}}{0}{}{01}\cancl{27} && +\hip{\rep{01},01,1,\rep{01},012,\rep{02}}{0}{}{01}\cancl{28}
\\&&+\hip{\rep{01},0,012,\rep{02}}{0}{}{01}\cancl{27} && +\hip{\rep{01},01,12,\rep{02}}{0}{}{01}\cancl{18}
\\&&+\hip{\rep{01},012,2,\rep{02}}{0}{}{01}\cancl{30} && 
\\&&+\hip{\rep{01},012,\rep{02},0,02,\rep{02}}{0}{}{01}\cancl{29} && +\hip{\rep{01},012,\rep{02},02,2,\rep{02}}{0}{}{01}\cancl{30}
\\&&+\hip{\rep{01},012,\rep{02},0}{0}{}{01}\cancl{29} && +\hip{\rep{01},012,\rep{02}}{0}{0}{01}\cancl{32}
\\&&+\hip{\rep{01},012,\rep{02}}{0}{0}{01}\cancl{32} && +\hip{\rep{01},012,\rep{02}}{0}{01}{1}\cancl{33}
\\&&+\hip{1,\rep{01},012,\rep{02}}{0}{}{01}\cancl{28} && +\hip{01,\rep{01},012,\rep{02}}{0}{}{0}\cancl{19}
\\
\delt_0 (\hip{\rep{01},012,\rep{02}}{01}{}{1})=\hspace{-5cm}
\\&& \hip{\rep{01},0,\rep{01},012,\rep{02}}{01}{}{1}\cancl{34}  && +\hip{\rep{01},1,\rep{01},012,\rep{02}}{01}{}{1}\cancl{35}  
\\&&+ \hip{\rep{01},01,\rep{02}}{01}{}{1}\cancl{23}  && +\hip{\rep{01},02,\rep{02}}{01}{}{1}\cancl{23}  
\\&&+ \hip{\rep{01},12,\rep{02}}{01}{}{1}\cancl{26}  && 
\\&&+ \hip{\rep{01},012,\rep{02},0,\rep{02}}{01}{}{1}\cancl{36}  && +\hip{\rep{01},012,\rep{02},2,\rep{02}}{01}{}{1}\cancl{37}  
\\&&+ \hip{\rep{01},012,\rep{02}}{0}{}{1}\cancl{31}  && +\hip{\rep{01},012,\rep{02}}{1}{}{1}\cancl{38}  
\\
\delt_1 (\hip{\rep{01},012,\rep{02}}{01}{}{1})=\hspace{-5cm}
\\&&\hip{\rep{01},0,01,\rep{01},012,\rep{02}}{01}{}{1}\cancl{34} && +\hip{\rep{01},01,1,\rep{01},012,\rep{02}}{01}{}{1}\cancl{35}
\\&&+\hip{\rep{01},0,012,\rep{02}}{01}{}{1}\cancl{34} && +\hip{\rep{01},01,12,\rep{02}}{01}{}{1}\cancl{26}
\\&&+\hip{\rep{01},012,2,\rep{02}}{01}{}{1}\cancl{37} && 
\\&&+\hip{\rep{01},012,\rep{02},0,02,\rep{02}}{01}{}{1}\cancl{36} && +\hip{\rep{01},012,\rep{02},02,2,\rep{02}}{01}{}{1}\cancl{37}
\\&&+\hip{\rep{01},012,\rep{02},0}{01}{}{1}\cancl{36} && +\hip{\rep{01},012,\rep{02},01}{1}{}{1}\cancl{39}
\\&&+\hip{\rep{01},012,\rep{02}}{01}{1}{1}\cancl{40} && +\hip{\rep{01},012,\rep{02}}{0}{01}{1}\cancl{33}
\\&&+\hip{\rep{01},012,\rep{02}}{01}{1}{1}\cancl{40} && +\hip{1,\rep{01},012,\rep{02}}{01}{}{1}\cancl{35}
\\
\delt_0 (\hip{\rep{02},012,\rep{12}}{1}{}{12})=\hspace{-5cm}
\\&& \hip{\rep{02},0,\rep{02},012,\rep{12}}{1}{}{12}\cancl{41}  && +  \hip{\rep{02},2,\rep{02},012,\rep{12}}{1}{}{12}\cancl{42}  
\\&&+ \hip{\rep{02},01,\rep{12}}{1}{}{12}\cancl{24}  && +  \hip{\rep{02},02,\rep{12}}{1}{}{12}\cancl{22}  
\\&&+ \hip{\rep{02},12,\rep{12}}{1}{}{12}\cancl{22}  &&
\\&&+ \hip{\rep{02},012,\rep{12},1,\rep{12}}{1}{}{12}\cancl{43}  && +  \hip{\rep{02},012,\rep{12},2,\rep{12}}{1}{}{12}\cancl{44}  
\\&&+ \hip{\rep{02},012,\rep{12}}{1}{}{1}\cancl{38}  && +  \hip{\rep{02},012,\rep{12}}{1}{}{2}\cancl{45}  
\\
\delt_1 (\hip{\rep{02},012,\rep{12}}{1}{}{12})=\hspace{-5cm}
\\&&\hip{\rep{02},0,02,\rep{02},012,\rep{12}}{1}{}{12}\cancl{41} && +\hip{\rep{02},02,2,\rep{02},012,\rep{12}}{1}{}{12}\cancl{42}
\\&&+\hip{\rep{02},0,012,\rep{12}}{1}{}{12}\cancl{41} && +\hip{\rep{02},01,12,\rep{12}}{1}{}{12}\cancl{24}
\\&&+\hip{\rep{02},012,2,\rep{12}}{1}{}{12}\cancl{44} && 
\\&&+\hip{\rep{02},012,\rep{12},1,12,\rep{12}}{1}{}{12}\cancl{43} && +\hip{\rep{02},012,\rep{12},12,2,\rep{12}}{1}{}{12}\cancl{44}
\\&&+\hip{\rep{02},012,\rep{12},1}{1}{}{12}\cancl{43} && +\hip{\rep{02},012,\rep{12}}{1}{1}{12}\cancl{46}
\\&&+\hip{\rep{02},012,\rep{12}}{1}{1}{12}\cancl{46} && +\hip{\rep{02},012,\rep{12}}{1}{12}{2}\cancl{47}
\\&&+\hip{2,\rep{02},012,\rep{12}}{1}{}{12}\cancl{42} && +\hip{12,\rep{02},012,\rep{12}}{1}{}{1}\cancl{48}
\\
\delt_0 (\hip{\rep{02},012,\rep{12}}{12}{}{2})=\hspace{-5cm}
\\&& \hip{\rep{02},0,\rep{02},012,\rep{12}}{12}{}{2}\cancl{49} && +\hip{\rep{02},2,\rep{02},012,\rep{12}}{12}{}{2}\cancl{50}
\\&&+ \hip{\rep{02},01,\rep{12}}{12}{}{2}\cancl{06} && +\hip{\rep{02},02,\rep{12}}{12}{}{2}\cancl{05}
\\&&+ \hip{\rep{02},12,\rep{12}}{12}{}{2}\cancl{05} && 
\\&&+ \hip{\rep{02},012,\rep{12},1,\rep{12}}{12}{}{2}\cancl{51} && +\hip{\rep{02},012,\rep{12},2,\rep{12}}{12}{}{2}\cancl{52}
\\&&+ \hip{\rep{02},012,\rep{12}}{1}{}{2}\cancl{45} && +\hip{\rep{02},012,\rep{12}}{2}{}{2}\cancl{07}
\\
\delt_1 (\hip{\rep{02},012,\rep{12}}{12}{}{2})=\hspace{-5cm}
\\&&\hip{\rep{02},0,02,\rep{02},012,\rep{12}}{12}{}{2}\cancl{49} && +\hip{\rep{02},02,2,\rep{02},012,\rep{12}}{12}{}{2}\cancl{50}
\\&&+\hip{\rep{02},0,012,\rep{12}}{12}{}{2}\cancl{49} && +\hip{\rep{02},01,12,\rep{12}}{12}{}{2}\cancl{06}
\\&&+\hip{\rep{02},012,2,\rep{12}}{12}{}{2}\cancl{52} && 
\\&&+\hip{\rep{02},012,\rep{12},1,12,\rep{12}}{12}{}{2}\cancl{51} && +\hip{\rep{02},012,\rep{12},12,2,\rep{12}}{12}{}{2}\cancl{52}
\\&&+\hip{\rep{02},012,\rep{12},1}{12}{}{2}\cancl{51} && +\hip{\rep{02},012,\rep{12},12}{2}{}{2}\cancl{07}
\\&&+\hip{\rep{02},012,\rep{12}}{12}{2}{2}\cancl{53} && +\hip{\rep{02},012,\rep{12}}{1}{12}{2}\cancl{47}
\\&&+\hip{\rep{02},012,\rep{12}}{12}{2}{2}\cancl{53} && +\hip{2,\rep{02},012,\rep{12}}{12}{}{2}\cancl{50}
\end{align*}
\begin{align*} %--------------------------------1 case--------------------------
\delt_0 (\hip{\rep{01},012,\rep{02},012,\rep{12}}{1}{}{1})=\hspace{-6.8cm}
\\&& \hip{\rep{01},0,\rep{01},012,\rep{02},012,\rep{12}}{1}{}{1}\cancl{54}
&&+\hip{\rep{01},1,\rep{01},012,\rep{02},012,\rep{12}}{1}{}{1}\cancl{55} 
\\&&+\hip{\rep{01},01,\rep{02},012,\rep{12}}{1}{}{1}\cancl{38} 
&&+\hip{\rep{01},02,\rep{02},012,\rep{12}}{1}{}{1}\cancl{38} 
\\&&+\hip{\rep{01},12,\rep{02},012,\rep{12}}{1}{}{1}\cancl{48} 
\\&&+\hip{\rep{01},012,\rep{02},0,\rep{02},012,\rep{12}}{1}{}{1}\cancl{56} 
&&+\hip{\rep{01},012,\rep{02},2,\rep{02},012,\rep{12}}{1}{}{1}\cancl{57} 
\\&&+\hip{\rep{01},012,\rep{02},01,\rep{12}}{1}{}{1}\cancl{39} 
&&+\hip{\rep{01},012,\rep{02},02,\rep{12}}{1}{}{1}\cancl{38} 
\\&&+\hip{\rep{01},012,\rep{02},12,\rep{12}}{1}{}{1}\cancl{38} 
\\&&+\hip{\rep{01},012,\rep{02},012,\rep{12},1,\rep{12}}{1}{}{1}\cancl{58} 
&&+\hip{\rep{01},012,\rep{02},012,\rep{12},2,\rep{12}}{1}{}{1}\cancl{59} 
\\
\delt_1 (\hip{\rep{01},012,\rep{02},012,\rep{12}}{1}{}{1})=\hspace{-6.8cm}
\\&& \hip{\rep{01},0,01,\rep{01},012,\rep{02},012,\rep{12}}{1}{}{1}\cancl{54} && +\hip{\rep{01},01,1,\rep{01},012,\rep{02},012,\rep{12}}{1}{}{1}\cancl{55}
\\&& +\hip{\rep{01},0,012,\rep{02},012,\rep{12}}{1}{}{1}\cancl{54} && +\hip{\rep{01},01,12,\rep{02},012,\rep{12}}{1}{}{1}\cancl{48}
\\&& +\hip{\rep{01},012,2,\rep{02},012,\rep{12}}{1}{}{1}\cancl{57} && 
\\&& +\hip{\rep{01},012,\rep{02},0,02,\rep{02},012,\rep{12}}{1}{}{1}\cancl{56} && +\hip{\rep{01},012,\rep{02},02,2,\rep{02},012,\rep{12}}{1}{}{1}\cancl{57}
\\&& +\hip{\rep{01},012,\rep{02},0,012,\rep{12}}{1}{}{1}\cancl{56} && +\hip{\rep{01},012,\rep{02},01,12,\rep{12}}{1}{}{1}\cancl{39}
\\&& +\hip{\rep{01},012,\rep{02},012,2,\rep{12}}{1}{}{1}\cancl{59} && 
\\&& +\hip{\rep{01},012,\rep{02},012,\rep{12},1,12,\rep{12}}{1}{}{1}\cancl{58} && +\hip{\rep{01},012,\rep{02},012,\rep{12},12,2,\rep{12}}{1}{}{1}\cancl{59}
\\&& +\hip{\rep{01},012,\rep{02},012,\rep{12},1}{1}{}{1}\cancl{58} && +\hip{\rep{01},012,\rep{02},012,\rep{12}}{1}{1}{1}\cancl{60}
\\&& +\hip{\rep{01},012,\rep{02},012,\rep{12}}{1}{1}{1}\cancl{60} && +\hip{1,\rep{01},012,\rep{02},012,\rep{12}}{1}{}{1}\cancl{55}
\end{align*}

In the above sum, there are several terms which appear twice and which therefore cancel (over $\Z_2$). The terms that are labeled with the following numbers appear twice and cancel:
\[
\cancl{03}, \cancl{08}, \cancl{09}, \cancl{10}, \cancl{15}, \cancl{16}, \cancl{17}, \cancl{25}, \cancl{31}, \cancl{32}, \cancl{33}, \cancl{40}, \cancl{45}, \cancl{46}, \cancl{47}, \cancl{53}, \cancl{60}.
\]
Moreover, in some instances the sum of three terms cancel. The following labels appear three times and also cancel:
\begin{align*}
 \cancl{01}, \cancl{02},  \cancl{06}, \cancl{07}, \cancl{11}, \cancl{12}, \cancl{18}, \cancl{19}, \cancl{20}, \cancl{21}, \cancl{24}, \cancl{26}, \cancl{27}, \cancl{28}, \cancl{29},
  \cancl{30}, \cancl{34}, \cancl{35},\\ \cancl{36},\cancl{37}, \cancl{39}, \cancl{41}, \cancl{42}, \cancl{43}, \cancl{44}, \cancl{48}, \cancl{49}, \cancl{50}, \cancl{51}, \cancl{52}, \cancl{54}, \cancl{55}, \cancl{56}, \cancl{57}, \cancl{58}, \cancl{59}.
 \end{align*}
There are six terms labeled with \cancl{38}, whose sum vanishes:
\begin{align*}
\hspace{0mm}\cancl{38}=&(\hip{\rep{01},012,\rep{02}}{1}{}{1}+\hip{\rep{01},012,\rep{02},02,\rep{12}}{1}{}{1}+\hip{\rep{01},012,\rep{02},12,\rep{12}}{1}{}{1})
\\& \hspace{-2mm}+ (\hip{\rep{02},012,\rep{12}}{1}{}{1}+\hip{\rep{01},01,\rep{02},012,\rep{12}}{1}{}{1} +\hip{\rep{01},02,\rep{02},012,\rep{12}}{1}{}{1})
\\ =& \hip{\rep{01},012,\rep{12}}{1}{}{1}+\hip{\rep{01},012,\rep{12}}{1}{}{1}=0
\end{align*}
The left-over terms that do not cancel are:
\begin{align*}
 \cancl{04} &=\hip{\rep{02}}{02}{}{2}\\
 \cancl{05} &=\hip{\rep{02}}{12}{}{2}+\hip{\rep{02},02,\rep{12}}{12}{}{2}+\hip{\rep{02},12,\rep{12}}{12}{}{2} =\hip{\rep{12}}{12}{}{2}\\
 \cancl{13}&=\hip{\rep{02}}{0}{}{01}+ \hip{\rep{01},01,\rep{02}}{0}{}{01}+\hip{\rep{01},02,\rep{02}}{0}{}{01}=\hip{\rep{01}}{0}{}{01}\\
  \cancl{14}&=\hip{\rep{02}}{0}{}{02}\\
  \cancl{22}&=\hip{\rep{02}}{1}{}{12}+ \hip{\rep{02},02,\rep{12}}{1}{}{12}+ \hip{\rep{02},12,\rep{12}}{1}{}{12}=\hip{\rep{12}}{1}{}{12}\\
  \cancl{23}&=\hip{\rep{02}}{01}{}{1}+ \hip{\rep{01},01,\rep{02}}{01}{}{1}+\hip{\rep{01},02,\rep{02}}{01}{}{1}=\hip{\rep{01}}{01}{}{1}
\end{align*}
These terms are precisely $\BF^{[02]}$, $\BF^{[12]}$, and $\BF^{[01]}$, so that $\delt\BF^{[012]}=\delt_0\BF^{[012]}+\delt_1\BF^{[012]}=\BF^{[01]}+\BF^{[02]}+\BF^{[12]}$, which is \eqref{EQU:D(F012)=F01+F02+F12}.

\end{document}